\documentclass[reqno, 12pt]{amsart}
\pdfoutput=1
\makeatletter
\let\origsection=\section \def\section{\@ifstar{\origsection*}{\mysection}}
\def\mysection{\@startsection{section}{1}\z@{.7\linespacing\@plus\linespacing}{.5\linespacing}{\normalfont\scshape\centering\S}}
\makeatother

\usepackage{amsmath,amssymb,amsthm}
\usepackage{mathrsfs}
\usepackage{mathabx}\changenotsign
\usepackage{dsfont}
\usepackage{afterpage}

\usepackage[normalem]{ulem}

\usepackage[dvipsnames]{xcolor}
\usepackage[backref]{hyperref}
\hypersetup{
    colorlinks,
    linkcolor={red!60!black},
    citecolor={green!60!black},
    urlcolor={blue!60!black}
}

\usepackage{graphicx}

\usepackage[open,openlevel=2,atend]{bookmark}

\usepackage[abbrev,msc-links,backrefs]{amsrefs}
\usepackage{doi}

\renewcommand{\PrintDOI}[1]{\doi{#1}}

\usepackage[T1]{fontenc}
\usepackage{lmodern}
\usepackage[babel]{microtype}
\usepackage[english]{babel}

\usepackage{geometry}
\numberwithin{equation}{section}
\numberwithin{figure}{section}

\usepackage{enumitem}
\def\rmlabel{\upshape({\itshape \roman*\,})}

\usepackage{pgf,tikz}
\usetikzlibrary{math,decorations.pathreplacing,calligraphy}
\usepackage{float}
\usepackage{longtable}

\AtBeginDocument{%
   \def\MR#1{}
}

\usepackage{hyperref}

\usepackage{caption}
\makeatletter
\def\greek#1{\expandafter\@greek\csname c@#1\endcsname}
\def\Greek#1{\expandafter\@Greek\csname c@#1\endcsname}
\def\@greek#1{\ifcase#1
	\or $\alpha$%
	\or $\beta$%
	\or $\gamma$%
	\or $\delta$%
	\or $\epsilon$%
	\or $\zeta$%
	\or $\eta$%
	\or $\theta$%
	\or $\iota$%
	\or $\kappa$%
	\or $\lambda$%
	\or $\mu$%
	\or $\nu$%
	\or $\xi$%
	\or $o$%
	\or $\pi$%
	\or $\rho$%
	\or $\sigma$%
	\or $\tau$%
	\or $\upsilon$%
	\or $\phi$%
	\or $\chi$%
	\or $\psi$%
	\or $\omega$%
\fi}
\def\@Greek#1{\ifcase#1
	\or $\mathrm{A}$%
	\or $\mathrm{B}$%
	\or $\Gamma$%
	\or $\Delta$%
	\or $\mathrm{E}$%
	\or $\mathrm{Z}$%
	\or $\mathrm{H}$%
	\or $\Theta$%
	\or $\mathrm{I}$%
	\or $\mathrm{K}$%
	\or $\Lambda$%
	\or $\mathrm{M}$%
	\or $\mathrm{N}$%
	\or $\Xi$%
	\or $\mathrm{O}$%
	\or $\Pi$%
	\or $\mathrm{P}$%
	\or $\Sigma$%
	\or $\mathrm{T}$%
	\or $\mathrm{Y}$%
	\or $\Phi$%
	\or $\mathrm{X}$%
	\or $\Psi$%
	\or $\Omega$%
\fi}

\AddEnumerateCounter{\greek}{\@greek}{24}
\AddEnumerateCounter{\Greek}{\@Greek}{12}

\makeatother

\let\polishlcross=\l
\def\l{\ifmmode\ell\else\polishlcross\fi}

\def\paragraph#1{%
  \noindent\textbf{#1.}\enspace}

\let\emptyset=\varnothing
\let\setminus=\smallsetminus

\makeatletter
\def\moverlay{\mathpalette\mov@rlay}
\def\mov@rlay#1#2{\leavevmode\vtop{   \baselineskip\z@skip \lineskiplimit-\maxdimen
   \ialign{\hfil$\m@th#1##$\hfil\cr#2\crcr}}}
\newcommand{\charfusion}[3][\mathord]{
    #1{\ifx#1\mathop\vphantom{#2}\fi
        \mathpalette\mov@rlay{#2\cr#3}
      }
    \ifx#1\mathop\expandafter\displaylimits\fi}
\makeatother

\DeclareFontFamily{U}  {MnSymbolC}{}
\DeclareSymbolFont{MnSyC}         {U}  {MnSymbolC}{m}{n}
\DeclareFontShape{U}{MnSymbolC}{m}{n}{
    <-6>  MnSymbolC5
   <6-7>  MnSymbolC6
   <7-8>  MnSymbolC7
   <8-9>  MnSymbolC8
   <9-10> MnSymbolC9
  <10-12> MnSymbolC10
  <12->   MnSymbolC12}{}
\DeclareMathSymbol{\powerset}{\mathord}{MnSyC}{180}

\usepackage{tikz}
\usetikzlibrary{calc,decorations.pathmorphing}
\pgfdeclarelayer{background}
\pgfdeclarelayer{foreground}
\pgfdeclarelayer{front}
\pgfsetlayers{background,main,foreground,front}

\let\epsilon=\varepsilon
\let\eps=\epsilon
\let\rho=\varrho
\let\theta=\vartheta
\let\kappa=\varkappa

\let\E=\EE

\def\PP{\Pr}

\def\rm{\mathbb{RM}}

\newcommand{\cM}{\mathcal{M}}

\newcommand{\cP}{\mathcal{P}}

\theoremstyle{plain}
\newtheorem{thm}{Theorem}[section]
\newtheorem{theorem}[thm]{Theorem}

\newtheorem{prop}[thm]{Proposition}

\newtheorem{cor}[thm]{Corollary}
\newtheorem{c-cor}[thm]{Co-Corollary}
\newtheorem{lemma}[thm]{Lemma}

\theoremstyle{definition}
\newtheorem{rem}[thm]{Remark}

\newtheorem{conj}[thm]{Conjecture}
\newtheorem{prob}[thm]{Problem}

\usepackage{accents}

\let\phi=\varphi

\DeclareMathOperator{\dec}{dec}

\begin{document}

\title[Erd\H os-Szekeres for uniform matchings]{Erd\H os-Szekeres type Theorems for ordered uniform  matchings}

\author{Andrzej Dudek}
\address{Department of Mathematics, Western Michigan University, Kalamazoo, MI, USA}
\email{\tt andrzej.dudek@wmich.edu}
\thanks{The first author was supported in part by Simons Foundation Grant MPS-TSM-00007551.}

\author{Jaros\l aw Grytczuk}
\address{Faculty of Mathematics and Information Science, Warsaw University of Technology, Warsaw, Poland}
\email{j.grytczuk@mini.pw.edu.pl}
\thanks{The second author was supported in part by Narodowe Centrum Nauki, grant 2020/37/B/ST1/03298.}

\author{Andrzej Ruci\'nski}
\address{Department of Discrete Mathematics, Adam Mickiewicz University, Pozna\'n, Poland}
\email{\tt rucinski@amu.edu.pl}
\thanks{The third author was supported in part by Narodowe Centrum Nauki, grant 2018/29/B/ST1/00426}

\begin{abstract} For $r,n\ge2$, an ordered $r$-uniform matching of size $n$ is an $r$-uniform hypergraph on a linearly ordered vertex set~$V$, with $|V|=rn$, consisting of $n$ pairwise disjoint edges. There are $\tfrac12\binom{2r}r$ different  ways two edges may intertwine,  called here patterns. Among them we identify $3^{r-1}$ collectable patterns $P$, which have the potential of appearing in arbitrarily large quantities called  $P$-cliques. 
	
We prove an Erd\H{o}s-Szekeres type result guaranteeing in \emph{every} ordered $r$-uniform matching the presence of a $P$-clique of a prescribed size, for \emph{some} collectable pattern~$P$. In particular, in the diagonal case, one of the $P$-cliques must be of size  $\Omega\left( n^{3^{1-r}}\right)$. In addition, for \emph{each} collectable pattern $P$ we show that the largest size of a $P$-clique in a \emph{random} ordered $r$-uniform matching of size $n$ is, with high probability, $\Theta\left(n^{1/r}\right)$.
\end{abstract}

\maketitle



\section{Introduction}\label{intro}

\subsection{Ordered matchings, words, and patterns}
A hypergraph $G$ is called \emph{ordered} if its vertex set is linearly ordered. If all edges of $G$ have the same size $r\ge 2$, then $G$ is an \emph{ordered $r$-uniform hypergraph}, or shortly, an \emph{ordered $r$-graph}. Let $G$ and $H$ be two ordered  $r$-graphs with $V(G)=\{v_1<\cdots<v_m\}$ and $V(H)=\{w_1<\cdots<w_m\}$, for some  $m\ge1$. We say that $G$ and $H$ are \emph{order-isomorphic} if for all $1\le i_1<\cdots<i_r\le m$, we have $\{v_{i_1},\dots,v_{i_r}\}\in E(G)$ if and only if $\{w_{i_1},\dots,w_{i_r}\}\in E(H)$.

Compared to ordinary isomorphism, one context in which order-isomorphism makes quite a difference is that of sub-hypergraph containment. If $G$ is an ordered $r$-graph, then any sub-$r$-graph $G'$ of $G$ can be also treated as an ordered $r$-graph with the ordering of $V(G')$ inherited from the ordering of $V(G)$.
Given  two ordered $r$-graphs,  $G$ and  $H$, we say that a sub-$r$-graph $G'\subset G$ is an \emph{ordered copy} of $H$ in $G$ if $G'$ and $H$ are order-isomorphic.
All kinds of questions concerning sub-hypergraphs in unordered hypergraphs can be posed for ordered hypergraphs as well (see, e.g., \cites{Tardos,BGT}).

In this paper we focus exclusively on \emph{ordered $r$-uniform matchings}, or shortly, \emph{$r$-matchings}, which are ordered $r$-graphs with pairwise disjoint edges (and no isolated vertices). There are precisely
$\frac{(rn)!}{(r!)^n\, n!}$ $r$-matchings of size $n$ (on a fixed ordered vertex set of size $rn$, typically on $[rn]:=\{1,\dots,rn\}$). 
Let ${\mathcal M}_n^{(r)}$ denote the family of all of them.

A convenient way of representing ordered $r$-matchings is to use \emph{words}. Let $M\in {\mathcal M}_n^{(r)}$ be an ordered $r$-matching of size $n$ on vertex set $[rn]$. A word $W=w_1w_2\cdots w_{rn}$ \emph{represents} $M$ if for every edge $e\in M$ we have
$w_i=w_j$ for all $i,j\in e$ and $w_\neq w_j$ for all $i\in e$ and $j\not\in e$.
 For instance, if $r=3$ and $M=\{\{1,2,4\}, \{3,10,12\},\{5,7,11\},\{6,8,9\}\}$,
then, two examples of   words representing $M$ are  $AABACDCDDBCB$ and $CCECDFDFFEDE$. Obviously, there are infinitely many such words, however,  every one of them consists of four distinct letters, each repeated exactly three times.

In general, any two words, $U=u_1\cdots u_{rn}$ and $W=w_1\cdots w_{rn}$, representing a given matching $M\in {\mathcal M}_n^{(r)}$, must be \emph{equivalent} in the sense that $u_i=u_j$ if and only if $w_i=w_j$, for all $1\le i<j\le rn$. Clearly, every such word consists of $n$ distinct letters, each occurring exactly $r$ times. If an alphabet of size $n$ is fixed, which is natural to assume, then every matching $M\in {\mathcal M}_n^{(r)}$ has exactly $n!$ distinct representing words constituting an equivalence class in the above relation.

Throughout the paper
we will be identifying ordered matchings with the equivalence classes of their representing words. Typically, to represent a given ordered matching, we will choose a word whose letters occur in alphabetic order. However, in a few places it will be convenient or even necessary to abandon this convention. 

Ordered $r$-matchings of size two are called \emph{patterns} (or \emph{$r$-patterns} if the uniformity is to be emphasized). For  $r=2$, there are three distinct patterns represented by words $AABB$, $ABBA$, and $ABAB$, or words equivalent to them. We call them, respectively, an \emph{alignment}, a \emph{nesting}, and a \emph{crossing} (see Fig. \ref{ESZ1}). For $r=3$ there exist already ten distinct patterns (see Table \ref{table:relations}), and in general, for every $r\ge 2$, there are exactly $\frac{1}{2}\binom{2r}{r}$ of them. The $r$-pattern in which one edge is completely to the left of the other, represented by word $\underbrace{A\cdots A}_r\underbrace{B\cdots B}_r$ will be called, as in the case $r=2$, an \emph{alignment}.

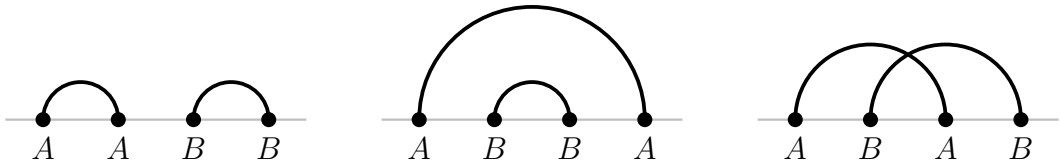
\begin{figure}[ht]
	\captionsetup[subfigure]{labelformat=empty}
	\begin{center}
		
		\scalebox{1}
		{
			\centering
			\begin{tikzpicture}
				[line width = .5pt,
				vtx/.style={circle,draw,black,very thick,fill=black, line width = 1pt, inner sep=0pt},
				]
				
				\coordinate (0) at (0.5,0) {};
				\node[vtx] (1) at (1,0) {};
				\node[vtx] (2) at (2,0) {};
				\node[vtx] (3) at (3,0) {};
				\node[vtx] (4) at (4,0) {};
				\coordinate (5) at (4.5,0) {};
				\draw[line width=0.3mm, color=lightgray]  (0) -- (5);
				\fill[fill=black, outer sep=1mm]  (1) circle (0.1) node [below] {$A$};
				\fill[fill=black, outer sep=1mm]  (2) circle (0.1) node [below] {$A$};
				\fill[fill=black, outer sep=1mm]  (3) circle (0.1) node [below] {$B$};
				\fill[fill=black, outer sep=1mm]  (4) circle (0.1) node [below] {$B$};
				\draw[line width=0.5mm, color=black, outer sep=2mm] (2) arc (0:180:0.5);
				\draw[line width=0.5mm, color=black, outer sep=2mm] (4) arc (0:180:0.5);
				
				\coordinate (0) at (5.5,0) {};
				\node[vtx] (1) at (6,0) {};
				\node[vtx] (2) at (7,0) {};
				\node[vtx] (3) at (8,0) {};
				\node[vtx] (4) at (9,0) {};
				\coordinate (5) at (9.5,0) {};
				\draw[line width=0.3mm, color=lightgray]  (0) -- (5);
				\fill[fill=black, outer sep=1mm]  (1) circle (0.1) node [below] {$A$};
				\fill[fill=black, outer sep=1mm]  (2) circle (0.1) node [below] {$B$};
				\fill[fill=black, outer sep=1mm]  (3) circle (0.1) node [below] {$B$};
				\fill[fill=black, outer sep=1mm]  (4) circle (0.1) node [below] {$A$};
				\draw[line width=0.5mm, color=black, outer sep=2mm] (3) arc (0:180:0.5);
				\draw[line width=0.5mm, color=black, outer sep=2mm] (4) arc (0:180:1.5);
				
				\coordinate (0) at (10.5,0) {};
				\node[vtx] (1) at (11,0) {};
				\node[vtx] (2) at (12,0) {};
				\node[vtx] (3) at (13,0) {};
				\node[vtx] (4) at (14,0) {};
				\coordinate (5) at (14.5,0) {};
				\draw[line width=0.3mm, color=lightgray]  (0) -- (5);
				\fill[fill=black, outer sep=1mm]  (1) circle (0.1) node [below] {$A$};
				\fill[fill=black, outer sep=1mm]  (2) circle (0.1) node [below] {$B$};
				\fill[fill=black, outer sep=1mm]  (3) circle (0.1) node [below] {$A$};
				\fill[fill=black, outer sep=1mm]  (4) circle (0.1) node [below] {$B$};
				\draw[line width=0.5mm, color=black, outer sep=2mm] (3) arc (0:180:1);
				\draw[line width=0.5mm, color=black, outer sep=2mm] (4) arc (0:180:1);
				
			\end{tikzpicture}
		}
		
	\end{center}
	\caption{An alignment, a nesting, and a crossing of a pair of edges.}
	\label{ESZ1}
	
\end{figure}

A homogeneous ordered $r$-matching, that is, one in which every pair of edges forms the same pattern $P$, will be called a \emph{$P$-clique}. For instance, the three possible $2$-patterns, $AABB$, $ABBA$, and $ABAB$, give rise, respectively, to three types of $P$-cliques, called \emph{lines}, \emph{stacks}, and \emph{waves} (see Fig.~\ref{ESZ2}). If the $r$-pattern $P$ is an alignment, then the $P$-clique will also be called a \emph{line}.

The term ``clique'' indicates that we will be dealing with a Ramsey type problem concerning monochromatic cliques in edge colored complete graphs. Indeed, for a given matching $M\in {\mathcal M}_n^{(r)}$, one may imagine a complete graph $K_n$ whose vertices are the edges of $M$ and whose edges are colored with patterns formed by their ends (pairs of edges of $M$). Then $P$-cliques in $M$ correspond exactly to monochromatic cliques of ``color'' $P$ in $K_n$.

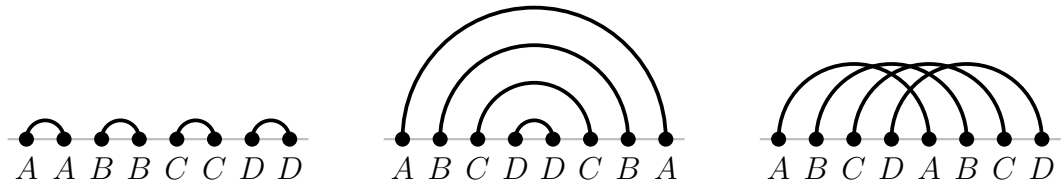
\begin{figure}[ht]
	\captionsetup[subfigure]{labelformat=empty}
	\begin{center}
		
		\scalebox{1}
		{
			\centering
			\begin{tikzpicture}
				[line width = .5pt,
				vtx/.style={circle,draw,black,very thick,fill=black, line width = 1pt, inner sep=0pt},
				]
				
				\coordinate (0) at (0.25,0) {};
				\node[vtx] (1) at (0.5,0) {};
				\node[vtx] (2) at (1,0) {};
				\node[vtx] (3) at (1.5,0) {};
				\node[vtx] (4) at (2,0) {};
				\node[vtx] (5) at (2.5,0) {};
				\node[vtx] (6) at (3,0) {};
				\node[vtx] (7) at (3.5,0) {};
				\node[vtx] (8) at (4,0) {};
				\coordinate (9) at (4.25,0) {};
				\draw[line width=0.3mm, color=lightgray]  (0) -- (9);
				\fill[fill=black, outer sep=1mm]  (1) circle (0.1) node [below] {$A$};
				\fill[fill=black, outer sep=1mm]  (2) circle (0.1) node [below] {$A$};
				\fill[fill=black, outer sep=1mm]  (3) circle (0.1) node [below] {$B$};
				\fill[fill=black, outer sep=1mm]  (4) circle (0.1) node [below] {$B$};
				\fill[fill=black, outer sep=1mm]  (5) circle (0.1) node [below] {$C$};
				\fill[fill=black, outer sep=1mm]  (6) circle (0.1) node [below] {$C$};
				\fill[fill=black, outer sep=1mm]  (7) circle (0.1) node [below] {$D$};
				\fill[fill=black, outer sep=1mm]  (8) circle (0.1) node [below] {$D$};
				\draw[line width=0.5mm, color=black, outer sep=2mm] (2) arc (0:180:0.25);
				\draw[line width=0.5mm, color=black, outer sep=2mm] (4) arc (0:180:0.25);
				\draw[line width=0.5mm, color=black, outer sep=2mm] (6) arc (0:180:0.25);
				\draw[line width=0.5mm, color=black, outer sep=2mm] (8) arc (0:180:0.25);
				
				\coordinate (0) at (5.25,0) {};
				\node[vtx] (1) at (5.5,0) {};
				\node[vtx] (2) at (6,0) {};
				\node[vtx] (3) at (6.5,0) {};
				\node[vtx] (4) at (7,0) {};
				\node[vtx] (5) at (7.5,0) {};
				\node[vtx] (6) at (8,0) {};
				\node[vtx] (7) at (8.5,0) {};
				\node[vtx] (8) at (9,0) {};
				\coordinate (9) at (9.25,0) {};
				\draw[line width=0.3mm, color=lightgray]  (0) -- (9);
				\fill[fill=black, outer sep=1mm]  (1) circle (0.1) node [below] {$A$};
				\fill[fill=black, outer sep=1mm]  (2) circle (0.1) node [below] {$B$};
				\fill[fill=black, outer sep=1mm]  (3) circle (0.1) node [below] {$C$};
				\fill[fill=black, outer sep=1mm]  (4) circle (0.1) node [below] {$D$};
				\fill[fill=black, outer sep=1mm]  (5) circle (0.1) node [below] {$D$};
				\fill[fill=black, outer sep=1mm]  (6) circle (0.1) node [below] {$C$};
				\fill[fill=black, outer sep=1mm]  (7) circle (0.1) node [below] {$B$};
				\fill[fill=black, outer sep=1mm]  (8) circle (0.1) node [below] {$A$};
				\draw[line width=0.5mm, color=black, outer sep=2mm] (5) arc (0:180:0.25);
				\draw[line width=0.5mm, color=black, outer sep=2mm] (6) arc (0:180:0.75);
				\draw[line width=0.5mm, color=black, outer sep=2mm] (7) arc (0:180:1.25);
				\draw[line width=0.5mm, color=black, outer sep=2mm] (8) arc (0:180:1.75);
				
				\coordinate (0) at (10.25,0) {};
				\node[vtx] (1) at (10.5,0) {};
				\node[vtx] (2) at (11,0) {};
				\node[vtx] (3) at (11.5,0) {};
				\node[vtx] (4) at (12,0) {};
				\node[vtx] (5) at (12.5,0) {};
				\node[vtx] (6) at (13,0) {};
				\node[vtx] (7) at (13.5,0) {};
				\node[vtx] (8) at (14,0) {};
				\coordinate (9) at (14.25,0) {};
				\draw[line width=0.3mm, color=lightgray]  (0) -- (9);
				\fill[fill=black, outer sep=1mm]  (1) circle (0.1) node [below] {$A$};
				\fill[fill=black, outer sep=1mm]  (2) circle (0.1) node [below] {$B$};
				\fill[fill=black, outer sep=1mm]  (3) circle (0.1) node [below] {$C$};
				\fill[fill=black, outer sep=1mm]  (4) circle (0.1) node [below] {$D$};
				\fill[fill=black, outer sep=1mm]  (5) circle (0.1) node [below] {$A$};
				\fill[fill=black, outer sep=1mm]  (6) circle (0.1) node [below] {$B$};
				\fill[fill=black, outer sep=1mm]  (7) circle (0.1) node [below] {$C$};
				\fill[fill=black, outer sep=1mm]  (8) circle (0.1) node [below] {$D$};
				\draw[line width=0.5mm, color=black, outer sep=2mm] (5) arc (0:180:1);
				\draw[line width=0.5mm, color=black, outer sep=2mm] (6) arc (0:180:1);
				\draw[line width=0.5mm, color=black, outer sep=2mm] (7) arc (0:180:1);
				\draw[line width=0.5mm, color=black, outer sep=2mm] (8) arc (0:180:1);
				
			\end{tikzpicture}
		}
		
	\end{center}
	
	\caption{A line, a stack, and a wave of size four.}
	\label{ESZ2}
	
\end{figure}

\subsection{Erd\H{o}s-Szekers type results for ordered $2$- and $3$-matchings}

 A frequent theme in combinatorics and graph theory concerns  unavoidable sub-structures that appear in every member of a prescribed family of structures. A flagship example  is the famous theorem of Erd\H{o}s and Szekeres \cite{ErdosSzekeres} on monotone subsequences (see, e.g., \cites{BKP,BucicSudakovTran,BukhMatousek,EliasMatousek,FoxPachSudakovSuk,MoshkovitzShapira,SzaboTardos,Czabarka} for some recent extensions and generalizations). In its diagonal form it states that any sequence $x_1,x_2,\ldots, x_n$ of distinct real numbers contains a monotone subsequence of length at least~$\sqrt n$. Our goal is to prove its analog for ordered $r$-matchings.

 The reason why the original  Erd\H os-Szekeres Theorem lists only two types of subsequences is, obviously, that for any two elements $x_i$ and $x_j$ with $i<j$ there are just two possible relations: $x_i< x_j$ or $x_i> x_j$. However, as noted above, for every two edges in an ordered $r$-matching, there are exactly $\tfrac12\binom{2r}r$ patterns in which they may intertwine. Therefore, more complex phenomena involving homogeneous sub-structures can be expected in this case.

 In \cite{DGR-match} (see also~\cite{DGR_LATIN}) we proved the following analog of the Erd\H os-Szekeres Theorem for ordered 2-matchings.

 \begin{theorem}[\cite{DGR-match}]\label{Theorem E-S for LSW}
	Let $\ell,s,w$ be arbitrary positive integers and let $n\ge\ell sw+1$. Then, every matching $M\in {\mathcal M}_n^{(2)}$ contains a line of size $\ell+1$, or a stack of size $s+1$, or a wave of size $w+1$. Moreover, this is not true for $n=\ell sw$.
\end{theorem}
In the symmetric case ($\ell=s=w$) Theorem~\ref{Theorem E-S for LSW} was deduced from Dilworth's Theorem by Huynh, Joos and Wollan in \cite[Lemma 21]{HJW2019} (though, a similar proof works in the general case too). For the sake of completeness, we give yet another proof
in the appendix.

 Also in \cite{DGR-match} we proved  an analogous, though a bit blurred, result for ordered $3$-matchings. There are $\tfrac12\binom63=10$ possible patterns that can be formed by two triples, see Table~\ref{table:relations} (ignore for a while the last column, as well as the alternative notation with double subscripts). Unfortunately, one of them, $P_1^*=AABABB$, stands out as one that cannot be formed mutually by more than two edges, in other words, there are no $P_1^*$-cliques of size larger than two. Nevertheless, in \cite{DGR-match} we managed to prove an Erd\H os-Szekeres type result guaranteeing in every $M\in\mathcal{M}_n^{(3)}$ the presence of $P$-clique of a prescribed size for one of the nine remaining patterns $P$, with some deficiency with respect to the pattern $P_1=AAABBB$ caused by a possible, though scattered, occurrence  of the defective pattern~$P_1^*$.

 \begin{table}[h!]
 	\centering
 	\begin{tabular}{ |c||c||c| }
 		\hline
 		
 		\textsc{pattern}& \textsc{as word} & \textsc{decomposition}\\
 		\hline
 		\hline
 		$P_1=P_{\ell,\ell}$ & $AAABBB$ & $\text{dec}(\underline{A}\overline{\underline{A}}\overline{A}\underline{B}\overline{\underline{B}}\overline{B})=(AABB,AABB)$\\
 		\hline
 		$P_1^*=P^*_{\ell,\ell}$ & $AABABB$ & $\text{dec}(\underline{A}\underline{\overline{A}}\underline{B}\overline{A}\underline{\overline{B}}\overline{B})=(AABB, AABB)$\\
 		\hline
 		\hline
 		$P_2=P_{\ell,s}$ & $AABBBA$ &  $\text{dec}(\underline{A}\overline{\underline{A}}\underline{B}\overline{\underline{B}}\overline{B}\overline{A})=(AABB,ABBA)$\\
 		\hline
 		$P_3=P_{\ell,w}$ & $AABBAB$ &  $\text{dec}(\underline{A}\overline{\underline{A}}\underline{B}\overline{\underline{B}}\overline{A}\overline{B})=(AABB,ABAB)$\\
 		\hline
 		$P_4=P_{s,\ell}$ & $ABBBAA$ & $\text{dec}(\underline{A}\underline{B}\overline{\underline{B}}\overline{B}\overline{\underline{A}}\overline{A})=(ABBA,BBAA)$\\
 		\hline
 		$P_5=P_{s,s}$ & $ABBAAB$ &  $\text{dec}(\underline{A}\underline{B}\overline{\underline{B}}\overline{\underline{A}}\overline{A}\overline{B})=(ABBA,BAAB)$\\
 		\hline
 		$P_6=P_{s,w}$ & $ABBABA$ &  $\text{dec}(\underline{A}\underline{B}\overline{\underline{B}}\overline{\underline{A}}\overline{B}\overline{A})=(ABBA,BABA)$\\
 		\hline
 		$P_7=P_{w,\ell}$  & $ABAABB$ & $\text{dec}(\underline{A}\underline{B}\overline{\underline{A}}\overline{A}\overline{\underline{B}}\overline{B})=(ABAB,AABB)$\\
 		\hline
 		$P_8=P_{w,s}$ & $ABABBA$ & $\text{dec}(\underline{A}\underline{B}\underline{\overline{A}}\underline{\overline{B}}\overline{B}\overline{A})=(ABAB,ABBA)$\\
 		\hline
 		$P_9=P_{w,w}$ & $ABABAB$ & $\text{dec}(\underline{A}\underline{B}\underline{\overline{A}}\underline{\overline{B}}\overline{A}\overline{B})=(ABAB,ABAB)$\\
 		\hline
 	\end{tabular}
 	\caption{All possible patterns of two triples and their corresponding decompositions. The rightmost column reveals the left parent (underlined) and the right parent (overlined) within the pattern. The alternative notation $P_{xy}$ encodes the pattern's parents, where $\ell$ stands for an alignment, $s$ -- a stack, and $w$ -- a wave.}
 	\label{table:relations}
 \end{table}

Recall that an ordered $r$-matching whose all pairs of edges form the same, fixed pattern $P$ is called a $P$-clique.
 \begin{theorem}[\cite{DGR-match}]\label{Theorem E-S for triples}
 	Let $a_1, a_2,\ldots, a_9$ be arbitrary positive integers and let $n\ge\prod_{i=1}^{9}a_{i}+1$.  Then, every matching $M\in{\mathcal M}_n^{(3)}$ either contains a $P_1$-clique of size at least $(a_1+1)/2$, or a $P_{i}$-clique of size $a_{i}+1$, for some $i\in\{2,3,\dots,9\}$.
 \end{theorem}
 \noindent As it turned out, this result is not optimal (see discussion in Section \ref{opt}).



\subsection{Collectable patterns and the main result}
In this subsection we describe our new results about unavoidable sub-structures in \emph{every} matching $M\in{\mathcal M}_n^{(r)}$. (For results about \emph{random} $r$-matchings see the next subsection.) The main goal of this paper is to generalize Theorems \ref{Theorem E-S for LSW} and \ref{Theorem E-S for triples} to arbitrary $r\ge 2$. 

In order to state our main result we need to introduce some terminology and notation. First we need to recognize patterns that allow for constructing arbitrarily large cliques. A pattern $P$ is called \emph{collectable} if for every $k\ge2$ it is possible to build a $P$-clique of size $k$. It turns out that collectable patterns are easily characterized in terms of their representing words.

   \emph{A block} in a word is a segment of consecutive letters of any length. A block consisting of the same letter is called a \emph{run} (it does not need to be maximal, though). A run $AA\cdots A$ of length $k$ is often denoted compactly by $A^k$ and called an \emph{$A$-run} of \emph{length} $k$.

  A pattern $P$ is \emph{splittable} if it can be split into a number of blocks each consisting of an $A$-run and a $B$-run  of the same length. For example,  $AABBABBA$ is splittable, because it splits into three such blocks: $|AABB|AB|BA|$, while $AABABB$ is not.

 It is easy to see that any splittable pattern is collectable. For instance, if $P=|AABB|AB|BA|$, then for every $k\ge 2$,  matching $$|A_1A_1A_2A_2\cdots A_kA_k|A_1A_2\cdots A_k|A_k\cdots A_2 A_1|$$
 is a $P$-clique of size $k$.
 Indeed, one can see that for every $1\le i<j\le k$, letters $A_i$ and $A_j$ do form pattern $P$.

 In Subsection~\ref{char} we show that a pattern is collectable if and only if it is splittable. It follows,  by means of some elementary enumeration, that there are precisely $3^{r-1}$ collectable $r$-patterns for every $r\ge2$.
 In particular,  for $r=2$ all three patterns are collectable, for $r=3$ all but one ($P_1^*$), while for $r=4$, out of 35 patterns, 27 are collectable (see Table \ref{table:patterns4} in Subsection \ref{pat_decomp}, where suitable partitions are shown).

 While  generalizing Theorems \ref{Theorem E-S for LSW} and \ref{Theorem E-S for triples} one has to accommodate the deficiencies  similar to the sole  constant $1/2$ in Theorem \ref{Theorem E-S for triples}. For this sake, we look at the very end of each pattern.
 The \emph{maturity} of a collectable pattern equals the length of its last \emph{maximal} run minus $2$, unless this value is negative, in which case we set it to $0$. The maturity of a pattern $P$ will be denoted by $m(P)$. For instance,  $m(AAAABBBB)=4-2=2$, while $m(AABBBBAA)=m(ABABABAB)=0$.

 We may now formulate our main result. An $r$-matching is called \emph{clean} if every pair of its edges forms a collectable pattern.

 \begin{theorem}\label{Theorem ESz General_classic}
	For $r\ge 2$, let $P_1,\dots,P_{3^{r-1}}$ be all collectable $r$-patterns and let positive integers $a_1,\dots, a_{3^{r-1}}$ be given. If $n\ge\prod_{i=1}^{3^{r-1}}a_i+1$, then
\begin{itemize}
\item[(a)] every matching $M\in{\mathcal M}_n^{(r)}$ contains a $P_i$-clique of size greater than $a_i2^{-m(P_i)}$, for some $i\in [3^{r-1}]$;
\item[(b)] every clean matching $M\in{\mathcal M}_n^{(r)}$ contains a $P_i$-clique of size at least $a_i+1$, for some $i\in [3^{r-1}]$.
\end{itemize}
\end{theorem}
\noindent Note that for $r=2$ and $r=3$, we retain, respectively, Theorems \ref{Theorem E-S for LSW} and \ref{Theorem E-S for triples}, the latter with an additional statement (b), where \emph{clean} just means that no pair of edges forms pattern $P_1^*$. As shown in Section \ref{opt}, Theorem~\ref{Theorem ESz General_classic} is not optimal.

\subsection{Corollaries} In this subsection we gather some easy corollaries of Theorem \ref{Theorem ESz General_classic} and prove them right away.

It can be routinely calculated (see Subsection \ref{char}) that the total maturity of all $r$-patterns equals
 \begin{equation}\label{summPi}
 \gamma_r:=\sum_{i=1}^{3^{r-1}}m(P_i)=\tfrac12\left(3^{r-2}-1\right).
  \end{equation}
  Thus, one can evenly redistribute the deficiencies $2^{-m(P_i)}$ occurring in Theorem \ref{Theorem ESz General_classic}(a) among \emph{all} collectable patterns. To avoid minor technical issues with integrality, we state it in the following form.

  \begin{cor}\label{ESz General_evenly}
	For $r\ge 2$, let $P_1,\dots,P_{3^{r-1}}$ be all collectable $r$-patterns and let positive integers $b_1,\dots, b_{3^{r-1}}$ be given. If $n>2^{\gamma_r}\prod_{i=1}^{3^{r-1}}b_i$, then
 every matching $M\in{\mathcal M}_n^{(r)}$ contains a $P_i$-clique of size at least $b_i+1$, for some $i\in [3^{r-1}]$.
 \end{cor}

\proof Set $a_i:=2^{m(P_i)}b_i$, $i=1,\dots,3^{r-1}$. Then
$$\prod_{i=1}^{3^{r-1}}a_i\overset{\eqref{summPi}}{=}2^\gamma\prod_{i=1}^{3^{r-1}}b_i<n$$ by assumption, and so, by  Theorem~\ref{Theorem ESz General_classic}(a), every matching $M\in{\mathcal M}_n^{(r)}$ contains, for some $i$, a $P_i$-clique of size greater than $a_i2^{-m(P_i)}=b_i$. \qed

\medskip

One context in which clean $r$-matchings emerge is that of $r$-partiteness. An ordered $r$-matching $M$ of size $k$ is \emph{$r$-partite} if, after splitting  its vertex set  into $r$ consecutive blocks   of size $k$,
 every edge of $M$ contains one vertex from each block.
Note that the only patterns possible in an $r$-partite $r$-matching are those which are themselves $r$-partite. These patterns can be characterized as splittable patterns with each block being either $AB$ or $BA$. Consequently, there are exactly $2^{r-1}$ of them.
For instance, for $r=2$, both  nesting and  crossing are bipartite, while
for $r=3$, only patterns $P_5,P_6,P_8,P_9$ (as listed in Table \ref{table:relations}) are tripartite.
Moreover, an $r$-partite $r$-matching is, indeed, clean.
Setting $a_i=1$ for each non-$r$-partite pattern $P_i$, we thus obtain the following corollary of Theorem~\ref{Theorem ESz General_classic}(b). Its optimality follows by a construction given in Section \ref{opt}.

\begin{cor}\label{Theorem ESz General_partite}
	For $r\ge 2$, let $Q_1,\dots,Q_{2^{r-1}}$ be all $r$-partite patterns and let positive integers $a_1,\dots, a_{2^{r-1}}$ be given. If $n\ge\prod_{i=1}^{2^{r-1}}a_i+1$, then every $r$-partite matching $M\in{\mathcal M}_n^{(r)}$ contains a $Q_i$-clique of size at least $a_i+1$, for some $i\in[2^{r-1}]$. Moreover, this is not true for $n=\prod_{i=1}^{2^{r-1}}a_i$. \qed
 \end{cor}
\noindent The above result illustrates the following principle: setting  $a_i=1$ for any pattern absent from a matching results in more ``prey'' for the remaining  patterns. One should also remark that Corollary~\ref{Theorem ESz General_partite} follows easily by $r-1$ applications of the original Erd\H os-Szekeres Theorem on permutations. Indeed, an $r$-partite $r$-matching of size $n$ can be viewed as an $(r-1)$-tuple of permutations of order $n$, while $P$-cliques, in this setting, correspond to $(r-1)$-tuples of monotone subsequences thereof.


 Theorem \ref{Theorem ESz General_classic} and its corollaries can be interpreted in terms of edge-colorings of a complete graph. As mentioned above, we may associate with every matching $M\in{\mathcal M}_n^{(r)}$ an edge colored clique $K_n$, whose vertices are the edges of $M$ and colors (of the edges of $K_n$) correspond to patterns occurring in $M$. 
 It follows from known bounds on Ramsey numbers that a monochromatic clique of order $\Omega(\log n)$ is guaranteed (see, e.g.,~\cite{ConlonFerber}), however, in this setting much larger cliques are present: of order $\Omega(n^{2^{1-r}})$ in  $r$-partite matchings and  $\Omega(n^{3^{1-r}})$ in general and clean matchings.

\begin{cor}\label{Theorem ESz General_diagonal}
	For $r\ge 2$, let $P_1,\dots,P_{3^{r-1}}$ be all collectable patterns and $Q_1,\ldots, Q_{2^{r-1}}$ be the $r$-partite patterns among them.
\begin{itemize}
\item[(a)] There is a constant $C_r>0$ such that every matching $M\in{\mathcal M}_n^{(r)}$ contains a $P_i$-clique of size at least $C_rn^{3^{1-r}}$, for some $i\in [3^{r-1}]$.
\item[(b)] Every clean matching $M\in{\mathcal M}_n^{(r)}$ contains a $P_i$-clique of size at least $n^{3^{1-r}}$, for some $i\in [3^{r-1}]$.
\item[(c)] Every $r$-partite matching $M\in{\mathcal M}_n^{(r)}$ contains a $Q_i$-clique of size greater than $n^{2^{1-r}}$, for some $i\in[2^{r-1}]$.
\end{itemize}
 \end{cor}

\begin{proof} For (a),  recall the definition of $\gamma_r$ in~\eqref{summPi} and set
$C_r:=2^{\gamma_r/3^{r-1}}$
 and $b_i=\lceil C_rn^{3^{1-r}}\rceil-1$, $i=1,\dots,3^{r-1}$. Then, $n>2^{\gamma_r}\prod_{i=1}^{3^{r-1}}b_i$ and, by Corollary~\ref{ESz General_evenly},   every $M\in{\mathcal M}_n^{(r)}$ contains, for some $i$, a $P_i$-clique of size greater than $b_i$, that is, at least $b_i+1\ge C_rn^{3^{1-r}}$.
 For (b), setting $a_i=\lceil n^{3^{1-r}}\rceil-1$, we deduce the conclusion from Theorem \ref{Theorem ESz General_classic}(b). Similarly, part (c) follows  from  Corollary \ref{Theorem ESz General_partite} by setting $a_i=\lceil n^{2^{1-r}}\rceil-1$.
\end{proof}

 \subsection{Random matchings}

In the second part of this paper we consider a \emph{random}  $r$-matching $\rm^{(r)}_{n}$, that is, a matching picked uniformly at random from the family ${\mathcal M}_n^{(r)}$ of all $\frac{(rn)!}{(r!)^n\, n!}$
$r$-matchings on the same vertex set~$[rn]$. We say that an event ${\mathcal A}_n$ holds \emph{asymptotically almost surely}, or shortly, \emph{a.a.s.}, if $\Pr({\mathcal A}_n)\to1$, as $n\to\infty$.

With respect to the  Erd\H os-Szekeres Theorem on monotone subsequences, it is well-known that a.a.s.\ a random permutation contains both, an increasing and a decreasing subsequence, of order $\Theta(\sqrt n)$. This, roughly, matches the order of a monotone subsequence guaranteed in every permutation. In \cite[Theorem 12 and Proposition 11]{DGR-match} we gave a simple proof of a similar result about  random 2-matchings. (For much sharper versions of this result see \cite{BaikRains} and \cite{JSW}.)

\begin{thm}[\cite{DGR-match}]\label{thm:random2}
A.a.s.\ the size of the largest line,  stack, and  wave  in a random matching $\rm_n^{(2)}$ is~$\Theta(\sqrt n)$.
\end{thm}

Note that, this time, the size of the sub-structure expected in a random matching exceeds that guaranteed in  every matching (c.f. Theorem \ref{Theorem E-S for LSW} with $\ell=s=w=\lceil n^{1/3}\rceil-1$ or Corollary~\ref{Theorem ESz General_diagonal}(b) for $r=2$). This trend continues with $r$ growing.
Indeed, in Section~\ref{section:random} we show that, for \emph{every} collectable $r$-pattern $P$,  the largest size of a $P$-clique in $\rm_n^{(r)}$ exceeds substantially the size  guaranteed in the deterministic case by Corollary \ref{Theorem ESz General_diagonal}.

 \begin{thm}\label{thm:random}
Let $P$ be an arbitrary collectable $r$-pattern. Then, a.a.s.\ the size of the largest $P$-clique in a random $r$-matching $\rm^{(r)}_{n}$ is~$\Theta_r(n^{1/r})$.
\end{thm}
Here and throughout, we use notation $\Theta_r(.)$, $O_r(.)$, and $\Omega_r(.)$ to stress that the hidden constant depends on~$r$.

\section{Proof of Theorem \ref{Theorem ESz General_classic}}\label{mainproof}

In this section, after some preparations, we prove Theorem \ref{Theorem ESz General_classic}.

\subsection{Collectable patterns}\label{char}

In Section \ref{intro} we defined collectable and splittable patterns. Now we show that these two notions are equivalent. In fact, for any non-collectable pattern~$P$, one cannot even construct a $P$-clique of order as small as three.

\begin{prop}\label{Proposition Collectable-Splitable}
	A pattern $P$ is collectable if and only if it is splittable. Moreover, if $P$ is unsplittable, then every $P$-clique has size at most two.
\end{prop}

\begin{proof}
	It is not hard to verify that every splittable pattern is collectable. Indeed, as demonstrated in Section \ref{intro}, it suffices to replace every block of the splitting, which is always of the form $A^tB^t$ or $B^tA^t$, by, respectively, $A_1^t\cdots A_k^t$ or $A_k^t\cdots A_1^t$.

	Now, assume that $P$ is an unsplittable pattern and let $Q$ be the longest splittable prefix of $P$, that is $P=QS$, where $Q$ is a splittable word (possibly empty) of length $2q$, $q\ge0$, and $S$ begins, say, with a block $A^sB^tA$  where $s>t\ge1$.  So, 
\begin{equation}\label{Pprop1}\textit{ there are exactly $q+t$ $B$'s to the left of the $(q+s+1)$-st letter $A$.}
 \end{equation}
 Moreover, there are $q+\min(s,2t)$  $A$'s and $q+\max(2t-s,0)$  $B$'s among the first $2q+2t$ letters of $P$. These precise numbers are of no importance to us - all what matters in the argument below is that the two numbers are different:
\begin{equation}\label{Pprop}
\textit{the numbers of  $A$'s and $B$'s among the first $2q+2t$ letters of $P$ are \emph{not} the same.}
\end{equation}
	
	Suppose to the contrary that there is a $P$-clique $K$ with three edges $e_A$, $e_B$ and $e_C$ (in lexicographic order) represented in the word notation, respectively, by letters $A,B$ and $C$. 
We now examine separately the two subwords $e_A\cup e_B$ and $e_A\cup e_C$, both isomorphic to $P$.
In view of~\eqref{Pprop1},  there are exactly $q+t$ letters $B$ and $q+t$ letters $C$ prior to the $(q+s+1)$-st letter $A$.
This implies that among the first $2q+2t$ letters of the subword $e_B\cup e_C$, the numbers of $B$'s and $C$'s are the same (both equal $q+t$).  This, however, contradicts  Property \eqref{Pprop} of $P$ observed earlier.

In the symmetric case, when $S$ begins with a block $B^sA^tB$, we examine instead subwords $e_A\cup e_C$ and $e_B\cup e_C$, obtaining the same contradiction for $e_A\cup e_B$.
\end{proof}

Next, using the above characterization, we enumerate the collectable $r$-patterns.
\begin{cor}\label{number_coll}
There are $3^{r-1}$ collectable patterns. Moreover, the total sum of their maturities equals $(3^{r-2}-1)/2$.
\end{cor}

\begin{proof} By Proposition \ref{Proposition Collectable-Splitable} we enumerate splittable patterns instead. Let $P$ be a splittable $r$-pattern. Note that its partition into appropriate blocks $X_1,\ldots, X_j$ is unique. Further,  the sizes $2s_i:=|X_i|$, $i=1,\dots,j$, of the partition blocks are determined by  ordered partitions  of $r=s_1+\dots+s_j$ into $j$  integer parts $s_i\ge1$ (which are then doubled). Finally, given such a partition,  a splittable pattern is obtained by deciding for each $i=2,\dots,j$,  which letter, $A$ or $B$, goes first in block $X_i$. (We assume that $X_1$ begins with $A$.) As there are $\binom{r-1}{j-1}$ such partitions, we infer that the number of splittable, and thus collectable $r$-patterns  equals
$$\sum_{j=1}^r\binom{r-1}{j-1}2^{j-1}=3^{r-1}.$$

For the second assertion,  let the last block in the splittable  partition of $P$ be $A^tB^t$, with $3\leq t\leq r$ (for $t=1$ and $t=2$, $m(P)=0$). For $t=r$, we have $P=A^rB^r$ and $m(P)=r-2$. For $3\leq t\leq r-1$, by the first statement of Corollary~\ref{number_coll}, there are $3^{r-t-1}$ collectable patterns with this ending block and the same number of them in the symmetric case when $P$ ends with  $B^tA^t$. Thus, there are exactly $2\cdot 3^{r-t-1}$ patterns with maturity $t-2$, for each $3\leq t \leq r-1$. Consequently, the total sum of maturities equals
$$(r-2)+2\cdot\sum_{t=3}^{r-1}(t-2)\cdot 3^{r-t-1}=(r-2)+2\cdot\sum_{m=1}^{r-3}m\cdot 3^{r-m-3}=\frac{1}{2}\cdot (3^{r-2}-1),$$
since $\sum_{m=1}^{r-3}m\cdot 3^{-m}=\frac{3}{4}\cdot [(3-2r) 3^{2-r}+1]$ (see, e.g.,\ identity (2.26) in~\cite{Concrete}).
\end{proof}

\subsection{Pattern decompositions and $P$-families}\label{pat_decomp}
In this subsection we define a  decomposition of patterns that will play a crucial role in the proof of our main result.

 Given an $r$-pattern $P$, $r\ge3$, let $Q:=Q(P)$ denote the $(r-1)$-pattern obtained from $P$ by deleting the last letter $A$ and the last letter $B$. In other words, $Q(P)$ is formed of the first $r-1$ letters $A$ and the first $r-1$ letters $B$ of $P$. Also, let $R:=R(P)$ be the $2$-pattern formed by the last two letters $A$ and the last two letters $B$ of $P$. In this way we have made a \emph{decomposition} of a given $r$-pattern $P$ into a pair of shorter patterns $(Q,R)$. We write  then $\dec(P)=(Q,R)$. For example, if $P=AABABBAB$, then $\dec(P)=(AABABB,\;ABAB)$. 
We call $Q$ and $R$, the \emph{left parent} and the \emph{right parent} of $P$, respectively, while $P$ is dubbed a \emph{child} of $Q$ and $R$ (see Fig.\ \ref{ESZ5}).

Recall that there are only three possible right parents, an alignment $R_1=AABB$, a nesting $R_2=ABBA$, and a crossing $R_3=ABAB$. However,  they may sometimes appear in the dual form of, respectively, $BBAA$, $BAAB$, and $BABA$, forced by the ordering of the letters in the child. 


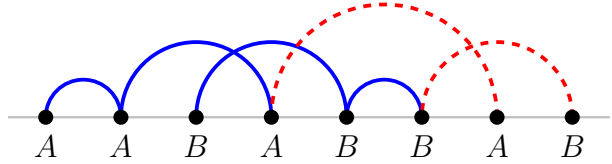
\begin{figure}[ht]
\captionsetup[subfigure]{labelformat=empty}
\begin{center}

\scalebox{1}
{
\centering
\begin{tikzpicture}
[line width = .5pt,
vtx/.style={circle,draw,black,very thick,fill=black, line width = 1pt, inner sep=0pt},
]
    \coordinate (0) at (0.5,0) {};
    \node[vtx] (1) at (1,0) {};
    \node[vtx] (2) at (2,0) {};
    \node[vtx] (3) at (3,0) {};
    \node[vtx] (4) at (4,0) {};
    \node[vtx] (5) at (5,0) {};
    \node[vtx] (6) at (6,0) {};
    \node[vtx] (7) at (7,0) {};
    \node[vtx] (8) at (8,0) {};
    \coordinate (9) at (8.5,0) {};

    \draw[line width=0.5mm, color=blue, outer sep=2mm] (2) arc (0:180:0.5);
    \draw[line width=0.5mm, color=blue, outer sep=2mm] (4) arc (0:180:1);
    \draw[line width=0.5mm, color=blue, outer sep=2mm] (5) arc (0:180:1);
    \draw[line width=0.5mm, color=blue, outer sep=2mm] (6) arc (0:180:0.5);
    \draw[line width=0.5mm, color=red, outer sep=2mm, dashed] (7) arc (0:180:1.5);
    \draw[line width=0.5mm, color=red, outer sep=2mm, dashed] (8) arc (0:180:1);

    \draw[line width=0.3mm, color=lightgray]  (0) -- (9);

    \fill[fill=black, outer sep=1mm]   (1) circle (0.1) node [below] {$A$};
    \fill[fill=black, outer sep=1mm]   (2) circle (0.1) node [below] {$A$};
    \fill[fill=black, outer sep=1mm]  (3) circle (0.1) node [below] {$B$};
    \fill[fill=black, outer sep=1mm]   (4) circle (0.1) node [below] {$A$};
    \fill[fill=black, outer sep=1mm]   (5) circle (0.1) node [below] {$B$};
    \fill[fill=black, outer sep=1mm]   (6) circle (0.1) node [below] {$B$};
    \fill[fill=black, outer sep=1mm]   (7) circle (0.1) node [below] {$A$};
    \fill[fill=black, outer sep=1mm]   (8) circle (0.1) node [below] {$B$};




\end{tikzpicture}
}

\end{center}

\caption{Decomposition of $AABABBAB$ into the left parent (blue) and the right parent (dashed red): $\dec(AABABBAB)=(AABABB,ABAB)$.}
\label{ESZ5}
	
\end{figure}

Clearly, for every pattern $P$ its parents $(Q,R)$ are determined uniquely. In other words, \text{dec} is a well defined function. However, it is not an injection: a pair of parents may ``give birth'' to more than one child. 
 Thus, in principle the pre-image $\dec^{-1}(Q,R)$ is not a singleton, in which case the children are called \emph{siblings}. When $\dec^{-1}(Q,R)=\{P\}$, we will simply write $\dec^{-1}(Q,R)=P$.

 For $r=3$ we have just one  instance of non-singular preimage, namely $\dec(AAABBB)=\dec(AABABB)=(AABB,AABB)$  (see Table \ref{table:relations}),  so here $\dec^{-1}(AABB,AABB)=\{AAABBB,AABABB\}$ is a 2-element set of siblings. For $r=4$, parents $Q=AAABBB$ and $R=AABB$ have three children, $P_1=AAAABBBB$, $P_1^*=AAABABBB$, and $P_1^{**}=AAABBABB$, that is, $\dec^{-1}(Q,R)=\{P_1,P_1^*,P_1^{**}\}$. There are, moreover, three 2-element sets of siblings, $\dec^{-1}(ABBBAA,BBAA)=\{P_{10},P^*_{10}\}$, $\dec^{-1}(ABABAABB,AABB)=\{P_{19},P^*_{19}\}$, and $\dec^{-1}(AAABABB,AABB)=\{\bar P_{\ell,\ell,\ell}, \bar P^*_{\ell,\ell,\ell}\}$, the last one, with a non-collectable left parent (the right parent is always $R_1$), will be of no interest for us (see Table \ref{table:patterns4}.
Notice also that each of the $3\cdot9=27$ pairs of collectable parents listed in Table \ref{table:patterns4} has exactly one collectable child. When this unique collectable child has siblings, it will be termed a \emph{big brother}. Thus, there are three big brothers, not four.

\begin{longtable}{ |l||c||c| }
	\hline
	\textsc{pattern $P$}& \textsc{as word} & \textsc{decomposition} $\text{dec}(P)$\\
	\hline
	\hline
 $P_1=P_{\ell,\ell,\ell}$ $(bb)$& $|AAAABBBB|$ & $(AAABBB,AABB)$\\
	\hline
	 $P_1^*=P^*_{\ell,\ell,\ell}$ & $AAABABBB$ &  $(AAABBB,AABB)$\\
	\hline
	 $P_1^{**}=P^{**}_{\ell,\ell,\ell}$ & $AAABBABB$ &  $(AAABBB,AABB)$\\
\hline
	 $P_{2}=P_{\ell,\ell,s}$ & $|AAABBB|BA|$ &  $(AAABBB,ABBA)$\\
\hline
	 $P_{3}=P_{\ell,\ell,w}$ & $|AAABBB|AB|$ & $(AAABBB,ABAB)$\\
	\hline
	\hline
	 $P_4=P_{\ell,s,\ell}$ & $|AABB|BBAA|$ &  $(AABBBA,BBAA)$\\
\hline
	 $P_{5}=P_{\ell,s,s}$ & $|AABB|BA|AB|$ & $(AABBBA,BAAB)$\\
\hline
	 $P_{6}=P_{\ell,s,w}$ & $|AABB|BA|BA|$ &  $(AABBBA,BABA)$\\
\hline
	\hline
	 $P_7=P_{\ell,w,\ell}$ & $|AABB|AABB|$ &  $(AABBAB,AABB)$\\
\hline
	 $P_{8}=P_{\ell,w,s}$ & $|AABB|AB|BA|$ & $(AABBAB,ABBA)$\\
\hline
	 $P_{9}=P_{\ell,w,w}$ & $|AABB|AB|AB|$ &  $(AABBAB,ABAB)$\\
\hline
	\hline
	 $P_{10}=P_{s,\ell,\ell}$ $(bb)$& $|AB|BBBAAA|$ &  $(ABBBAA,BBAA)$\\
	\hline
	 $P_{10}^*=P^*_{s,\ell,\ell_1}$ & $ABBBABAA$ &  $(ABBBAA,BBAA)$\\
\hline
	 $P_{11}=P_{s,\ell,s}$ & $|AB|BBAA|AB|$ & $(ABBBAA,BAAB)$\\
\hline
	 $P_{12}=P_{s,\ell,w}$ & $|AB|BBAA|BA|$ &  $(ABBBAA,BABA)$\\
\hline
\hline
	 $P_{13}=P_{s,s,\ell}$ & $|AB|BA|AABB|$ &  $(ABBAAB,AABB)$\\
\hline
	 $P_{14}=P_{s,s,s}$ & $|AB|BA|AB|BA|$ & $(ABBAAB,ABBA)$\\
\hline
	 $P_{15}=P_{s,s,w}$ & $|AB|BA|AB|AB|$ &  $(ABBAAB,ABAB)$\\
\hline
\hline
	 $P_{16}=P_{s,w,\ell}$ & $|AB|BA|BBAA|$ & $(ABBABA,BBAA)$\\
\hline
	 $P_{17}=P_{s,w,s}$ & $|AB|BA|BA|AB|$ &  $(ABBABA,BAAB)$\\
\hline
	 $P_{18}=P_{s,w,w}$ & $|AB|BA|BA|BA|$ &  $(ABBABA,BABA)$\\
\hline
	\hline
 $P_{19}=P_{w,\ell,\ell}$ $(bb)$ & $|AB|AAABBB|$ & $(ABAABB,AABB)$\\
	\hline
	 $P_{19}^*=P^*_{w,\ell,\ell}$ & $ABAABABB$ & $(ABAABB,AABB)$\\
	\hline
	 $P_{20}=P_{w,\ell,s}$ & $|AB|AABB|BA|$ &  $(ABAABB,ABBA)$\\
	\hline
	 $P_{21}=P_{w,\ell,w}$ & $|AB|AABB|AB|$ &  $(ABAABB,ABAB)$\\
\hline
\hline
	 $P_{22}=P_{w,s,\ell}$ & $|AB|AB|BBAA|$ &  $(ABABBA,BBAA)$\\
\hline
	 $P_{23}=P_{w,s,s}$ & $|AB|AB|BA|AB|$ & $(ABABBA,BAAB)$\\
\hline
	 $P_{24}=P_{w,s,w}$ & $|AB|AB|BA|BA|$ &  $(ABABBA,BABA)$\\
\hline
\hline
	 $P_{25}=P_{w,w,\ell}$ & $|AB|AB|AABB|$ &  $(ABABAB,AABB)$\\
	\hline
	 $P_{26}=P_{w,w,s}$ & $|AB|AB|AB|BA|$ & $(ABABAB,ABBA)$\\
\hline
	 $P_{27}=P_{w,w,w}$ & $|AB|AB|AB|AB|$ &  $(ABABAB,ABAB)$\\
\hline
\hline
	 $\bar P_{\ell,\ell,\ell}$ & $AABAABBB$ &  $(AABABB,AABB)$\\
	\hline
	 $\bar P^*_{\ell,\ell,\ell}$ & $AABABABB$ & $(AABABB,AABB)$\\
	\hline
	 $\bar P_{\ell,\ell,w}$ & $AABABBAB$ &  $(AABABB,ABAB)$\\
	\hline
	 $\bar P_{\ell,\ell,s}$ & $AABABBBA$ &  $(AABABB,ABBA)$\\
	\hline
	\caption{All 35 patterns of two quadruples  and their corresponding decompositions. They are organized into 10 groups by the patterns of their left parents, while the alternative notation $P_{xyz}$ encodes the pattern's parents (cf.\ Table \ref{table:relations}).  Big brothers are marked with $(bb)$, their siblings with $^*$, while other non-collectable patterns (i.e., those whose left parents are non-collectable too) are put at the end, unnumbered, and marked with $\;\bar{}$.}
	\label{table:patterns4}
\end{longtable}

We collect all essential properties of the decomposition function $\dec(P)$ in Proposition \ref{Proposition Big Brother} below.
Let $t(Q)$ denote the length of the last (maximal) run in $Q$. Note that $t(Q)\ge1$ and, whenever $t(Q)\ge2$, we have $t(Q)=m(Q)+2$.

\begin{prop}\label{Proposition Big Brother} Let $r\ge 3$ and let $Q$ be a collectable $(r-1)$-pattern.
\begin{itemize}
\item[(i)] For each $j=2,3$, the pair $(Q,R_j)$ has only one child, and the same is true for the pair $(Q,R_1)$ provided $t(Q)=1$. Moreover, this only child is collectable and its maturity is equal to zero.
\item[(ii)] If $t(Q)\ge2$, then the pair $(Q,R_1)$ has $t(Q)$ children, only one of which, the big brother, is collectable. Moreover, the maturity of the big brother equals $t(Q)-1=m(Q)+1$.
\end{itemize}
\end{prop}

\begin{proof} Denote $t(Q)=t$ and choose a word representation of $Q$ which ends with letter $B$, that is, $Q=Q'A^tB^t$, for some $t\geq 1$, where $Q'$ is a collectable $(r-1-t)$-pattern.
	
To produce any child $P$ of the pair $(Q,R_j)$, $j=1,2,3$, one has to attach the $2$-pattern $R_j$  to the back of $Q$ so that applying the decomposition function $\dec(P)$ would yield back the parents $Q$ and $R_j$. More precisely, one has to identify the first $A$ and the first $B$ of $R_j$ with the last $A$ and the  last $B$ of $Q$, respectively, and then locate the second $A$ and the second $B$ of $R_j$ accordingly to pattern $R_j$. It is not hard to check that for $j\in\{2,3\}$ there is only one way to perform this operation, as illustrated in Fig.\ \ref{ESZ6}, and that we always have $P=\dec^{-1}(Q, R_2)=QAB$ and $P=\dec^{-1}(Q,R_3)=QBA$. Clearly, in both cases the pattern $P$ is collectable and satisfies $m(P)=0$.

\begin{figure}[ht]
\captionsetup[subfigure]{labelformat=empty}
\begin{center}

\scalebox{1}
{
\centering
\begin{tikzpicture}
[line width = .5pt,
vtx/.style={circle,draw,black,very thick,fill=black, line width = 1pt, inner sep=0pt},
]
    \coordinate (0) at (-0.25,0) {};
    \node[vtx] (1) at (1,0) {};
    \node[vtx] (2) at (2,0) {};
    \node[vtx] (3) at (3,0) {};
    \node[vtx] (4) at (4,0) {};
    \coordinate (44) at (4,1) {};
    \node[vtx] (5) at (5,0) {};
    \node[vtx] (6) at (6,0) {};
    \node[vtx] (7) at (7,0) {};
    \coordinate (8) at (7.5,0) {};

    \draw[line width=0.3mm, color=lightgray]  (0) -- (8);

    \draw[line width=0.5mm, color=blue, outer sep=2mm] (2) arc (0:180:0.5);
    \draw[line width=0.5mm, color=blue, outer sep=2mm] (4) arc (0:180:0.5);
    \draw[line width=0.5mm, color=blue, outer sep=2mm] (5) arc (0:180:0.5);
    \draw[line width=0.5mm, color=red, outer sep=2mm, dashed] (7) arc (0:180:1);
    \draw[line width=0.5mm, color=red, outer sep=2mm, dashed] plot [smooth, tension=2] coordinates {(2) (44) (6)};
    \draw[line width=0.5mm, color=blue, outer sep=2mm] (1) arc (0:85:1);
    \draw[line width=0.5mm, color=blue, outer sep=2mm] (3) arc (0:120:1);

    \fill[fill=black, outer sep=1mm]   (1) circle (0.1) node [below] {$A$};
    \fill[fill=black, outer sep=1mm]   (2) circle (0.1) node [below] {$A$};
    \fill[fill=black, outer sep=1mm]  (3) circle (0.1) node [below] {$B$};
    \fill[fill=black, outer sep=1mm]   (4) circle (0.1) node [below] {$B$};
    \fill[fill=black, outer sep=1mm]   (5) circle (0.1) node [below] {$B$};
    \fill[fill=black, outer sep=1mm]   (6) circle (0.1) node [below] {$A$};
    \fill[fill=black, outer sep=1mm]   (7) circle (0.1) node [below] {$B$};

    \coordinate (0) at (7.75,0) {};
    \node[vtx] (1) at (9,0) {};
    \node[vtx] (2) at (10,0) {};
    \node[vtx] (3) at (11,0) {};
    \node[vtx] (4) at (12,0) {};
    \coordinate (45) at (12.5,1) {};
    \node[vtx] (5) at (13,0) {};
    \node[vtx] (6) at (14,0) {};
    \node[vtx] (7) at (15,0) {};
    \coordinate (8) at (15.5,0) {};

    \draw[line width=0.5mm, color=blue, outer sep=2mm] (2) arc (0:180:0.5);
    \draw[line width=0.5mm, color=blue, outer sep=2mm] (4) arc (0:180:0.5);
    \draw[line width=0.5mm, color=blue, outer sep=2mm] (5) arc (0:180:0.5);
    \draw[line width=0.5mm, color=red, outer sep=2mm, dashed] (6) arc (0:180:0.5);
    \draw[line width=0.5mm, color=red, outer sep=2mm, dashed] plot [smooth, tension=2] coordinates {(2) (45) (7)};
    \draw[line width=0.5mm, color=blue, outer sep=2mm] (1) arc (0:85:1);
    \draw[line width=0.5mm, color=blue, outer sep=2mm] (3) arc (0:120:1);

    \draw[line width=0.3mm, color=lightgray]  (0) -- (8);

    \fill[fill=black, outer sep=1mm]   (1) circle (0.1) node [below] {$A$};
    \fill[fill=black, outer sep=1mm]   (2) circle (0.1) node [below] {$A$};
    \fill[fill=black, outer sep=1mm]  (3) circle (0.1) node [below] {$B$};
    \fill[fill=black, outer sep=1mm]   (4) circle (0.1) node [below] {$B$};
    \fill[fill=black, outer sep=1mm]   (5) circle (0.1) node [below] {$B$};
    \fill[fill=black, outer sep=1mm]   (6) circle (0.1) node [below] {$B$};
    \fill[fill=black, outer sep=1mm]   (7) circle (0.1) node [below] {$A$};

\end{tikzpicture}
}

\end{center}

\caption{The only children of $Q$ with $ABAB$ and with $ABBA$.}
\label{ESZ6}
	
\end{figure}
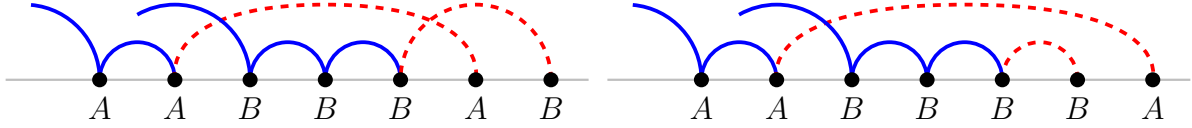

There is also a unique outcome whenever $j=t=1$, but in this case we have $P=\dec^{-1}(Q,R_1)=Q'AABB$ (see Fig.\ \ref{ESZ13}). Hence, the pattern $P$ is collectable and $m(P)=0$. This proves part (i) of the proposition.

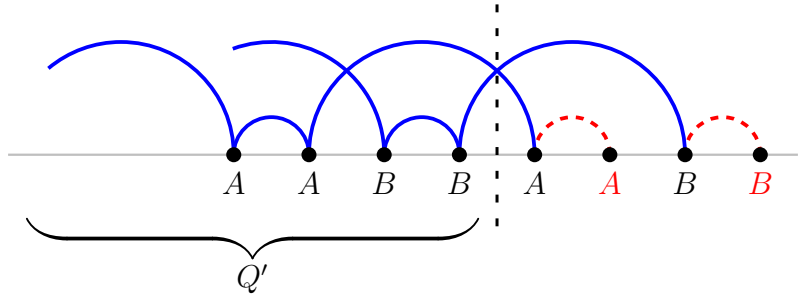
\begin{figure}[ht]
\captionsetup[subfigure]{labelformat=empty}
\begin{center}

\scalebox{1}
{
\centering
\begin{tikzpicture}
[line width = .5pt,
vtx/.style={circle,draw,black,very thick,fill=black, line width = 1pt, inner sep=0pt},
]
    \coordinate (0) at (0,0) {};
    \node[vtx] (1) at (3,0) {};
    \node[vtx] (2) at (4,0) {};
    \node[vtx] (3) at (5,0) {};
    \node[vtx] (4) at (6,0) {};
    \node[vtx] (5) at (7,0) {};
    \node[vtx] (6) at (8,0) {};
    \node[vtx] (7) at (9,0) {};
    \node[vtx] (8) at (10,0) {};
    \coordinate (9) at (10.5,0) {};

    \draw[line width=0.3mm, color=lightgray]  (0) -- (9);
    \draw[line width=0.4mm, color=black, loosely dashed]  (6.5, 2) -- (6.5,-1);

    \draw[line width=0.5mm, color=blue, outer sep=2mm] (2) arc (0:180:0.5);
    \draw[line width=0.5mm, color=blue, outer sep=2mm] (4) arc (0:180:0.5);
    \draw[line width=0.5mm, color=red, outer sep=2mm, dashed] (6) arc (0:180:0.5);
    \draw[line width=0.5mm, color=red, outer sep=2mm, dashed] (8) arc (0:180:0.5);
    \draw[line width=0.5mm, color=blue, outer sep=2mm] (5) arc (0:180:1.5);
    \draw[line width=0.5mm, color=blue, outer sep=2mm] (7) arc (0:180:1.5);
    \draw[line width=0.5mm, color=blue, outer sep=2mm] (1) arc (0:130:1.5);
    \draw[line width=0.5mm, color=blue, outer sep=2mm] (3) arc (0:110:1.5);
    \draw [line width=0.5mm, decorate, decoration = {calligraphic brace,mirror,raise=10pt, amplitude=15pt}] (0.25,-0.5) --  (6.25,-0.5);
    \node at (3.25,-1.65) {$Q'$};

    \fill[fill=black, outer sep=1mm]   (1) circle (0.1) node [below] {$A$};
    \fill[fill=black, outer sep=1mm]   (2) circle (0.1) node [below] {$A$};
    \fill[fill=black, outer sep=1mm]  (3) circle (0.1) node [below] {$B$};
    \fill[fill=black, outer sep=1mm]   (4) circle (0.1) node [below] {$B$};
    \fill[fill=black, outer sep=1mm]   (5) circle (0.1) node [below] {$A$};
    \fill[fill=black, outer sep=1mm]   (6) circle (0.1) node [below] {\textcolor{red}{$A$}};
    \fill[fill=black, outer sep=1mm]   (7) circle (0.1) node [below] {$B$};
    \fill[fill=black, outer sep=1mm]   (8) circle (0.1) node [below] {\textcolor{red}{$B$}};

\end{tikzpicture}
}

\end{center}

\caption{The only child of $Q$ with $AABB$ when $t=1$.}
\label{ESZ13}
	
\end{figure}

Now assume that $R=R_1$ and $t\ge2$. Obviously the second $B$ of $R_1$ must be attached at the very end of $Q$. However, the second $A$ of $R_1$ may be located either right before the ending run $B^t$, or somewhere within it, as long as the child ends with a run of at least two $B$'s. This brings  exactly $t$ different $r$-patterns --- all children of $Q$ and $R_1$ (see Fig.\ \ref{ESZ7}).

\begin{figure}[ht]
\captionsetup[subfigure]{labelformat=empty}
\begin{center}

\scalebox{1}
{
\centering
\begin{tikzpicture}
[line width = .5pt,
vtx/.style={circle,draw,black,very thick,fill=black, line width = 1pt, inner sep=0pt},
]

    \coordinate (0) at (0,0) {};
    \node[vtx] (1) at (1,0) {};
    \node[vtx] (2) at (2,0) {};
    \node[vtx] (3) at (3,0) {};
    \node[vtx] (4) at (4,0) {};
    \node[vtx] (5) at (5,0) {};
    \node[vtx] (6) at (6,0) {};
    \node[vtx] (7) at (7,0) {};
    \coordinate (8) at (7.5,0) {};
    \coordinate (11) at (1,0.7) {};
    \coordinate (21) at (2.5,1) {};

    \draw[line width=0.3mm, color=lightgray]  (0) -- (8);
    \draw[line width=0.5mm, color=blue, outer sep=2mm] (2) arc (0:180:0.5);
    \draw[line width=0.5mm, color=blue, outer sep=2mm] (5) arc (0:180:0.5);
    \draw[line width=0.5mm, color=blue, outer sep=2mm] (6) arc (0:180:0.5);
    \draw[line width=0.5mm, color=red, outer sep=2mm, dashed] (3) arc (0:180:0.5);
    \draw[line width=0.5mm, color=red, outer sep=2mm, dashed] (7) arc (0:180:0.5);
    \draw[line width=0.5mm, color=blue, outer sep=2mm] (1) arc (0:80:1);
    \draw[line width=0.5mm, color=blue, outer sep=2mm] plot [smooth, tension=1] coordinates {(11) (21) (4)};

    \fill[fill=black, outer sep=1mm]  (1) circle (0.1) node [below] {$A$};
    \fill[fill=black, outer sep=1mm]  (2) circle (0.1) node [below] {$A$};
    \fill[fill=black, outer sep=1mm]  (3) circle (0.1) node [below] {$A$};
    \fill[fill=black, outer sep=1mm]  (4) circle (0.1) node [below] {$B$};
    \fill[fill=black, outer sep=1mm]  (5) circle (0.1) node [below] {$B$};
    \fill[fill=black, outer sep=1mm]  (6) circle (0.1) node [below] {$B$};
    \fill[fill=black, outer sep=1mm]  (7) circle (0.1) node [below] {$B$};

(b)
    \coordinate (0) at (8,0) {};
    \node[vtx] (1) at (9,0) {};
    \node[vtx] (2) at (10,0) {};
    \node[vtx] (3) at (11,0) {};
    \node[vtx] (4) at (12,0) {};
    \node[vtx] (5) at (13,0) {};
    \node[vtx] (6) at (14,0) {};
    \node[vtx] (7) at (15,0) {};
    \coordinate (8) at (15.5,0) {};
    \coordinate (11) at (9,0.7) {};
    \coordinate (21) at (10,1) {};

    \draw[line width=0.3mm, color=lightgray]  (0) -- (8);
    \draw[line width=0.5mm, color=blue, outer sep=2mm] (2) arc (0:180:0.5);
    \draw[line width=0.5mm, color=blue, outer sep=2mm] (5) arc (0:180:1);
    \draw[line width=0.5mm, color=blue, outer sep=2mm] (6) arc (0:180:0.5);
    \draw[line width=0.5mm, color=red, outer sep=2mm, dashed] (4) arc (0:180:1);
    \draw[line width=0.5mm, color=red, outer sep=2mm, dashed] (7) arc (0:180:0.5);
    \draw[line width=0.5mm, color=blue, outer sep=2mm] (1) arc (0:80:1);
    \draw[line width=0.5mm, color=blue, outer sep=2mm] plot [smooth, tension=1] coordinates {(11) (21) (3)};

    \fill[fill=black, outer sep=1mm]  (1) circle (0.1) node [below] {$A$};
    \fill[fill=black, outer sep=1mm]  (2) circle (0.1) node [below] {$A$};
    \fill[fill=black, outer sep=1mm]  (3) circle (0.1) node [below] {$B$};
    \fill[fill=black, outer sep=1mm]  (4) circle (0.1) node [below] {$A$};
    \fill[fill=black, outer sep=1mm]  (5) circle (0.1) node [below] {$B$};
    \fill[fill=black, outer sep=1mm]  (6) circle (0.1) node [below] {$B$};
    \fill[fill=black, outer sep=1mm]  (7) circle (0.1) node [below] {$B$};

    \coordinate (0) at (3.75,-2.5) {};
    \node[vtx] (1) at (4.75,-2.5) {};
    \node[vtx] (2) at (5.75,-2.5) {};
    \node[vtx] (3) at (6.75,-2.5) {};
    \coordinate (34) at (7.25,-1.5) {};
    \node[vtx] (4) at (7.75,-2.5) {};
    \node[vtx] (5) at (8.75,-2.5) {};
    \node[vtx] (6) at (9.75,-2.5) {};
    \node[vtx] (7) at (10.75,-2.5) {};
    \coordinate (8) at (11.25,-2.5) {};
    \coordinate (11) at (4.75, -1.8) {};
    \coordinate (21) at (5.75,-1.5) {};

    \draw[line width=0.3mm, color=lightgray]  (0) -- (8);
    \draw[line width=0.5mm, color=blue, outer sep=2mm] (2) arc (0:180:0.5);
    \draw[line width=0.5mm, color=blue, outer sep=2mm] (4) arc (0:180:0.5);
    \draw[line width=0.5mm, color=blue, outer sep=2mm] (6) arc (0:180:1);
    \draw[line width=0.5mm, color=red, outer sep=2mm, dashed] (7) arc (0:180:0.5);
    \draw[line width=0.5mm, color=red, outer sep=2mm, dashed] plot [smooth, tension=2] coordinates {(2) (34) (5)};
    \draw[line width=0.5mm, color=blue, outer sep=2mm] (1) arc (0:80:1);
    \draw[line width=0.5mm, color=blue, outer sep=2mm] plot [smooth, tension=1] coordinates {(11) (21) (3)};

    \fill[fill=black, outer sep=1mm]  (1) circle (0.1) node [below] {$A$};
    \fill[fill=black, outer sep=1mm]  (2) circle (0.1) node [below] {$A$};
    \fill[fill=black, outer sep=1mm]  (3) circle (0.1) node [below] {$B$};
    \fill[fill=black, outer sep=1mm]  (4) circle (0.1) node [below] {$B$};
    \fill[fill=black, outer sep=1mm]  (5) circle (0.1) node [below] {$A$};
    \fill[fill=black, outer sep=1mm]  (6) circle (0.1) node [below] {$B$};
    \fill[fill=black, outer sep=1mm]  (7) circle (0.1) node [below] {$B$};

\end{tikzpicture}
}

\end{center}

\caption{Extending the last block of $Q$ by $AABB$ ($t=3$).}
\label{ESZ7}
	
\end{figure}
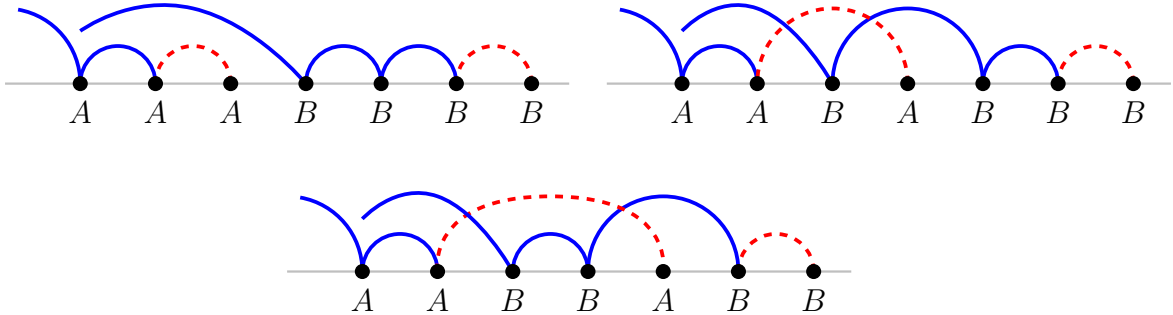
Formally,
\begin{equation*}
\dec^{-1}(Q,AABB)=\{Q'A^{t}B^{t_1}AB^{t_2}:\quad t_1+t_2=t+1,\; t_1\geq 0,\;t_2\geq 2\}.
\end{equation*}
Among these children only the one with $t_1=0$ (and thus, $t_2=t+1$) is collectable, because $Q'$ is. Moreover, its maturity equals $t+1-2=t-1=m(Q)+1$.  For $t_1\ge1$, there is a block $BAB^{t_2}$ at the end with $t_2\geq 2$, which excludes splittability, and thus collectability (cf.\ Proposition \ref{Proposition Collectable-Splitable}). 
This completes the proof of part (ii).
\end{proof}

\medskip

In view of Proposition \ref{Proposition Big Brother}, we may classify all collectable patterns into two categories: those that are the only child of their parents and those that have \emph{siblings}. In the latter case, as it was already mentioned earlier, we call them \emph{big brothers} (see Table \ref{table:patterns4} where the three big brothers are marked by (bb)). It follows from  Proposition \ref{Proposition Big Brother}(ii) that a big brother $P$ has a positive maturity $m(P)$ equal to the number of its siblings. For instance, the pattern $P_1=AAAABBBB$ with $m(P_1)=2$ is a big brother with two siblings, $P_1^*=AAABABBB$ and $P_1^{**}=AAABBABB$. It can be easily calculated that, for $r\ge3$, there are $3^{r-3}$ big brothers.

A big brother $P$ together with all its siblings constitute a \emph{$P$-family}. A $\emph{$P$-family-clique}$ is a matching whose every pair of edges forms a pattern belonging to the $P$-family. Our next result states that at least half of any $P$-family-clique makes up a pure $P$-clique. The special case of $r=3$  was already proved in \cite[Prop. 4]{DGR-match}.


\begin{prop}\label{Proposition P-family-clique}
 For $r\ge3$, let an $r$-pattern $P$ be a big brother. Then every  $P$-family-clique of size $k$ contains a  $P$-clique of size at least $k/2$.
\end{prop}

\begin{proof}
  Let an $r$-pattern $P$ be a big brother  with $\dec(P)=(Q,AABB)$ for some collectable $(r-1)$-pattern $Q$. We express $Q=Q'A^tB^t$ for some $t\ge2$ and a collectable $(r-t)$-pattern $Q'$ ($Q'=\emptyset$ when $t=r$). Let $K_P=\{e_1,\dots,e_k\}$ be a $P$-family-clique of size $k$, with edge $e_i$ represented by letter $A_i$, $i=1,\dots,k$. Further, let $K_Q$ be the $(r-1)$-matching obtained from $K_P$ by dropping last letter from each $e_i$. Then $K_Q$ is a $Q$-clique which, as $Q$ is collectable, ends in a block $T=A_1^{t}\cdots  A_k^{t}$ (w.l.o.g., we assume that the letters appear in that order).
   
   Now, for each $i=1,\dots,k-1$, the last (i.e., the $r$-th) letter $A_i$ of edge $e_i$ in $K_P$ must be located after the last letter $A_i$ in $T$ but before the last letter $A_{i+1}$ in $T$ (and the last letter $A_k$ must go at the very end), as otherwise edges $e_i, e_{i+1}$ would  form a pattern whose right parent is not $AABB$. Consequently, for all pairs $(i,j)$ with $j\ge i+2$, edges $e_i,e_j$ form a collectable pattern of the form $Q'A_i^{t+1}A_j^{t+1}$, so this pattern must be the big brother $P$ (and not any of its siblings).
  Hence, the set of  all edges $e_i\in K_P$ with odd $i$ forms a $P$-clique of size $\lceil k/2\rceil$.
\end{proof}

\subsection{The  inductive  proof}\label{section:hyper_new}

The proof of Theorem \ref{Theorem ESz General_classic} is by induction on $r$ with the base case of $r=2$ holding by Theorem \ref{Theorem E-S for LSW} proved in \cite[Theorem 1]{DGR-match}. (For the sake of completeness, we have reproved it in Appendix~\ref{appendix}.) For part (a), we found it convenient to prove an equivalent statement, Theorem \ref{Theorem ESz General_real} below, which uses real-valued parameters.
To allow parameters strictly smaller than one we adopt the convention that a single edge of an $r$-matching forms by itself a $P$-clique for any $r$-pattern $P$.

  \begin{thm}\label{Theorem ESz General_real}
	For $r\ge 2$, let $P_1,\dots,P_{3^{r-1}}$ be all collectable patterns and let positive real numbers $x_1,\dots, x_{3^{r-1}}$ be given. If $n>\prod_{i=1}^{3^{r-1}}x_i$, then every matching $M\in\cM_n^{(r)}$ contains a $P_i$-clique of size greater than
$x_i2^{-m(P_i)}$, for some $i\in[3^{r-1}]$.
\end{thm}
   To see the equivalence, let $a_i=\lfloor x_i\rfloor$, $i=1,\dots, 3^{r-1}$, and observe that $n\ge\prod_{i=1}^{3^{r-1}}a_i+1$. Thus, by Theorem~\ref{Theorem ESz General_classic}(a), for some $i$, there is a $P_i$-clique of size greater than $a_i2^{-m(P_i)}$, that is, of size at least $(a_i+1)2^{-m(P_i)}>x_i2^{-m(P_i)}$.
 Inversely, by choosing integer values for the numbers $x_i$, Theorem \ref{Theorem ESz General_real} trivially implies Theorem \ref{Theorem ESz General_classic}(a). So, the two statements are actually equivalent and, in particular, for $r=2$, Theorem \ref{Theorem ESz General_real} follows from Theorem \ref{Theorem E-S for LSW}.

\begin{proof}[Proof of Theorem \ref{Theorem ESz General_real}]
	Let $r\geq3$ and assume that the statement is true for $r-1$. Let $x_1,\ldots, x_{3^{r-1}}$ be positive real numbers and let $n>\prod_{i=1}^{3^{r-1}}x_i$ be an integer. Let $P_1,\ldots, P_{3^{r-1}}$ be all collectable patterns of uniformity $r$ and let $Q_1,\ldots, Q_{3^{r-2}}$ be all collectable patterns of uniformity $r-1$. We  assume, for convenience, that this enumeration is chosen, as in Tables 1 and 2, so that $\dec(P_{3(i-1)+j})=(Q_i,R_j)$, where $R_1=AABB$, $R_2=ABBA$, and $R_3=ABAB$ (or their dual versions). Let us also define $y_i=\prod_{j=1}^{3}x_{3(i-1)+j}$ for all $i=1,2,\ldots,3^{r-2}$. Notice that $\prod_{i=1}^{3^{r-2}}y_i=\prod_{i=1}^{3^{r-1}}x_i<n$.
	
	Consider  a matching $M\in\cM_n^{(r)}$. Let $M_1$ be the $(r-1)$-matching  obtained from $M$ by deleting the last vertex of every edge. Thus, $M_1\in\cM_n^{(r-1)}$. Apply the induction hypothesis to $M_1$ with the patterns $Q_i$ and the numbers $y_i$, $i=1,\dots,3^{r-2}$, defined above, obtaining, for some $i$, a $Q_i$-clique $K_i$ in $M_1$ of size $q_1>y_i/2^{m(Q_i)}$.
	
	Let us look now at the 2-matching $M_2$ formed by the pairs of the last two vertices of all edges of $M$ whose set of the first $r-1$ vertices belongs to $K_i$. Formally,
$$M_2=\big{\{}\{i_{r-1},i_r\}:\; \{i_1<\cdots<i_r\}\in M\quad\mbox{and}\quad \{i_1,\dots,i_{r-1}\}\in K_i\big{\}}.$$
Clearly, $M_2\in\cM_{q_1}^{(2)}$ and we may apply to it the case $r=2$ of Theorem \ref{Theorem ESz General_real} with
$$z_1=x_{3(i-1)+1}/2^{m(Q_i)}\quad\mbox{and}\quad z_j=x_{3(i-1)+j}\quad\mbox{for}\quad j=2,3$$
(note that $q_1>z_1z_2z_3$). Hence, for some $j\in \{1,2,3\}$, we get an $R_j$-clique $L_j$ in $M_2$ of size $q_2>z_j$.
	
Let us now extend the edges of $L_j$ back to the original $r$-edges of $M$, obtaining  a sub-matching $M_3$ of $M$ of size $q_2$.
Formally,
$$M_3=\big{\{}\{i_1<\cdots<i_r\}\in M:\;\{i_1,\dots,i_{r-1}\}\in K_i\quad\mbox{and}\quad\{i_{r-1},i_r\}\in L_j\big{\}}.$$

By the definition of function $\dec$, every pair of edges in $M_3$ forms a pattern $P$ which decomposes into $(Q_i,R_j)$. Assume first that $j\in\{2,3\}$, or $j=1$, but then $t(Q_i)=1$ and thus, $m(Q_i)=0$. In all these three cases, by our enumeration scheme defined at the beginning of the proof, $P=P_{3(i-1)+j}$. Moreover, by Proposition \ref{Proposition Big Brother}(i), the pattern $P$ is unique and, consequently, $M_3$ forms a $P$-clique of size  $q_2>z_j=x_{3(i-1)+j}$.

 In the remaining case, when $j=1$ and $t(Q_i)\ge2$, by Proposition \ref{Proposition Big Brother}(ii),  the pattern $P_{3(i-1)+1}$ is a big brother with $m(P_{3(i-1)+1})=t(Q_i)-1=m(Q_i)+1$. Then $M_3$ is a $P_{3(i-1)+1}$-family-clique of size $q_2>z_1$. By Proposition \ref{Proposition P-family-clique}, $M_3$ contains a $P_{3(i-1)+1}$-clique of size at least $$q_2/2>z_1/2=x_{3(i-1)+1}/2^{m(Q_i)+1}=x_{3(i-1)+1}/2^{m(P_{3(i-1)+1})},$$
 which completes the proof.
\end{proof}

The proof of part (b) of Theorem \ref{Theorem ESz General_classic} is very similar and even simpler, as there is no reason to resort to the real-valued version. Moreover, for a clean matching $M\in \mathcal{M}_n^{(r)}$ we do not need to worry about the maturities and the case $j=1$, $t(Q_i)\ge2$, does not differ from the others. For the induction step, however, it is crucial to observe that if $M$ is clean, then so is $M_1$ --- the $(r-1)$-matching obtained from $M$ by deleting the last vertex of every edge. Indeed, removing the last $A$ and the last $B$ from a splittable $r$-pattern results, obviously, in a splittable $(r-1)$-pattern.


\section{Proof of Theorem \ref{thm:random}}\label{section:random}
In this section we prove Theorem \ref{thm:random} which provides estimates of the  size of the largest $P$-clique one can find in a \emph{random} ordered $r$-uniform matching.
Recall that $\rm^{(r)}_{n}$ denotes a random (ordered) $r$-matching of size $n$, that is, an $r$-matching  picked uniformly at random out of the set of all
\[
\alpha^{(r)}_n:=\frac{(rn)!}{(r!)^n\, n!}
\]
matchings $M\in \mathcal{M}_n^{(r)}$ on the set $[rn]$. There are two other equivalent ways of drawing $\rm^{(r)}_{n}$, the \emph{permutational scheme} and the \emph{online scheme}.

 The formula for $\alpha^{(r)}_n$ indicates that each ordered $r$-matching can be coupled with exactly $(r!)^n n!$ permutations. Indeed, one can generate an ordered $r$-matching by the following permutational scheme. Let $\pi$ be a permutation of $[rn]$. We chop $\pi$  into an $r$-matching $\{\{\pi(1),\dots,\pi(r)\}, \{\pi(r+1),\dots,\pi(2r)\},\dots,\{\pi(rn-r-1),\dots,\pi(rn)\}\}$ and, clearly, there are exactly $(r!)^n n!$ permutations $\pi$ yielding the same matching. This scheme  allows one to use concentration inequalities for random permutations in the context of random matchings (see Subsection \ref{lb-gen}).

The online scheme of generating $\rm^{(r)}_{n}$ goes as follows. Given an arbitrary ordering of the vertices $u_1,\dots,u_{rn}$ (not necessarily the same as the canonical ordering $1,2,\dots,rn$) one selects uniformly at random   an $(r-1)$-element set $\{u_{j_1},\dots,u_{j_{r-1}}\}$ (in $\binom{rn-1}{r-1}$ ways) to be matched with $u_1$. Then, after crossing out $u_1,u_{j_1},\dots,u_{j_{r-1}}$ from the list, one selects uniformly at random an  $(r-1)$-element set (in $\binom{rn-r-1}{r-1}$ ways) to be matched with the first  uncrossed vertex, and so on, and so forth.  This scheme comes in handy when estimating probabilities of events involving small sets of fixed vertices (see Subsection \ref{lb-gen} and Appendix~C).

\subsection{Upper bound}

We begin with an upper bound. A slightly more general result was already stated in \cite[Prop. 11]{DGR-match} but here we present it again, just for $P$-cliques, for the sake of completeness.
\begin{lemma}\label{lem:upbr}
Let $P$ be a collectable $r$-pattern. Then, a.a.s.\ the size of the largest $P$-clique in a random $r$-matching $\rm^{(r)}_{n}$ is~$O_r(n^{1/r})$.
\end{lemma}

\begin{proof} For every integer $k\ge2$, let $X_k$ be the number of $P$-cliques of size $k$ in
$\rm^{(r)}_n$. In order to compute the expectation of $X_k$, one has to choose a set $S$ of $rk$ vertices (out of all $rn$ vertices) on which a $P$-clique will be planted. Formally, for each $S\in\binom{[rn]}{rk}$ define an indicator random variable $I_S=1$ if there is a $P$-clique on $S$ in $\rm^{(r)}_n$, and $I_S=0$ otherwise. As there is just one way to plant a $P$-clique on a given set $S$, we have  $X_k=\sum_SI_S$ and
\[
\E I_S = \Pr(I_S=1)=\frac{\alpha^{(r)}_{n-k}}{\alpha^{(r)}_n}.
\]
Thus, by the linearity of expectation,
	\[
	\E X_k = \binom{rn}{rk}  \cdot \frac{\alpha^{(r)}_{n-k}}{\alpha^{(r)}_n}
	= \frac{(r!)^k}{(rk)!} \cdot \frac{n!}{(n-k)!}
	\le \frac{(r!)^k}{(rk)!} \cdot n^k
	\le \frac{(r!)^k}{(rk/e)^{rk}} \cdot n^k
	= \left( \frac{e^r r! n}{(rk)^r} \right)^k.
	\]
Fixing a constant $c>\tfrac er(r!)^{1/r}$ and setting  $k_0=\lceil cn^{1/r}\rceil$, we conclude, by Markov's inequality, that
$$\Pr(X_{k_0}>0)\le\E X_{k_0}\le\left( \frac{e^r r! }{(rc)^r} \right)^{k_0}=o(1).$$
This completes the proof.
\end{proof}

\subsection{Lower bound}\label{lb-gen}
In this subsection we prove the lower bound in Theorem~\ref{thm:random} for all collectable patterns $P$. 
We start with a straightforward estimate which generalizes one from \cite[Lemma 2.2]{ConlonFoxSudakov}, though with a worse constant.
For given disjoint subsets $A_1,\dots,A_r \subseteq [rn]$ we say that an edge $e$ of an $r$-matching \emph{spans} these sets if $|e\cap A_j|=1$ for each $1\le j\le r$ (see Fig.\ \ref{span}).

\begin{figure}[ht]
\captionsetup[subfigure]{labelformat=empty}
\begin{center}

\scalebox{1}
{
\centering
\begin{tikzpicture}
[line width = .5pt,
vtx/.style={circle,draw,black,very thick,fill=black, line width = 1pt, inner sep=0pt},
]

    \node[vtx] (1) at (0,0) {};
    \node[vtx] (2) at (2,0) {};
    \node[vtx] (3) at (4,0) {};
    \coordinate (34) at (6,1) {};
    \node[vtx] (4) at (8,0) {};
    \node[vtx] (5) at (10,0) {};

    \draw[line width=0.3mm, color=lightgray]  (1) -- (5);

    \draw[line width=2.5mm, color=gray, line cap=round]  (1.5,0) -- (2.5,0);
    \draw[line width=2.5mm, color=gray, line cap=round]  (3.75,0) -- (4.75,0);
    \draw[line width=2.5mm, color=gray, line cap=round]  (7.25,0) -- (8.25,0);

    \node[color=gray] at (2.6,-0.35) {$A_1$};
    \node[color=gray] at (4.85,-0.35) {$A_2$};
    \node[color=gray] at (8.35,-0.35) {$A_3$};
    \node[color=blue] at (4,0.7) {$e$};

    \draw[line width=0.5mm, color=blue, outer sep=2mm] plot [smooth, tension=2] coordinates {(3) (34) (4)};
    \draw[line width=0.5mm, color=blue, outer sep=2mm] (3) arc (0:180:1);

    \fill[fill=black, outer sep=1mm]   (1) circle (0.1) node [below] {$1$};
    \fill[fill=black, outer sep=1mm]   (2) circle (0.1);
    \fill[fill=black, outer sep=1mm]  (3) circle (0.1);
    \fill[fill=black, outer sep=1mm]   (4) circle (0.1);
    \fill[fill=black, outer sep=1mm]   (5) circle (0.1) node [below] {$3n$};

\end{tikzpicture}
}

\end{center}

\caption{An edge spanning three sets.}
\label{span}
			
\end{figure}
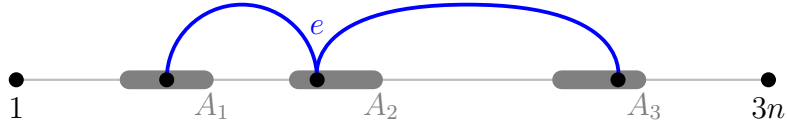

\begin{lemma}\label{lemma:span_prob} Given integers $r\ge2$ and  $2r\le t\le n$,
	let $A_1,\dots,A_r$ be disjoint subsets of $[rn]$ each of size $t$. Then, the probability that there are no edges in $\rm^{(r)}_{n}$ spanning all these sets is at most $\exp\left\{-\frac{1}{(2r)^r} \cdot \frac{t^r}{n^{r-1}}\right\}$.
\end{lemma}

\begin{proof} Let $\mathcal E$ be the event that no edge in $\rm^{(r)}_{n}$ spans the sets  $A_1,\dots,A_r$.
	 We will  generate $\rm^{(r)}_{n}$ using the online scheme with the vertices of $A_1$ coming first. We denote by $e_1,e_2,\dots$ the edges drawn randomly along this scheme. Let $t'=\lceil t/(2r)\rceil$ and ${\mathcal E}_i$, $1\le i\le t'$, be the event that $e_i$ does not span  $A_1,\dots,A_r$. Then, by the chain formula,
\begin{equation}\label{PPP}
\PP(\mathcal E)\le\PP\left(\bigcap_{i=1}^{t'}{\mathcal E}_i\right)=\PP({\mathcal E}_1)\PP({\mathcal E}_2|{\mathcal E}_1)\cdots\PP\left({\mathcal E}_{t'}\Big|\bigcap_{i=1}^{t'-1}{\mathcal E}_i\right).
\end{equation}

Let $a_1\in A_1$ be the first vertex to get its match in the online scheme. There are exactly $t^{r-1}$ $r$-element sets of vertices containing $a_1$ which do span the sets $A_1,\dots,A_r$. Thus, the probability that $e_1$ is not any of them
is 
	\[
	\PP({\mathcal E}_1)=1-\frac{t^{r-1}}{\binom{rn-1}{r-1}}.
	\]

	Note that $|A_1\setminus e_1|\ge t-r>0$ and let $a_2$ be the first vertex in $A_1\setminus e_1$. Since  $|A_i\cap e_1|\le r-1$ for each $j=2,\dots,r$, the number of $r$-element sets of vertices containing $a_2$, disjoint from $e_1$, and spanning the sets $A_1,\dots,A_r$ is at least $(t-r+1)^{r-1}$. Thus,

	\[
	\PP({\mathcal E}_2|{\mathcal E}_1)\le1-\frac{ (t-r+1)^{r-1}}{\binom{rn-r-1}{r-1}}.
	\]

	We repeat this process $t'$ times, matching randomly that many vertices from~$A_1$. Note that after $t'-1$ rounds, by the choice of $t'$,  the number of available vertices in each $A_j$, $j\ge2$, is  at least $t-(t'-1)(r-1)>t/2$. Hence, there is still a vertex $a_{t'}\in A_1\setminus(e_1\cup\cdots\cup e_{t'-1})$ and 
$$\PP\left({\mathcal E}_{t'}\Big|\bigcap_{i=1}^{t'-1}{\mathcal E}_i\right)\le1-\frac{ [t-(t'-1)(r-1)]^{r-1}}{\binom{rn-(t'-1)r-1}{r-1}}.$$
Hence, by \eqref{PPP},
	$$
		\PP(\mathcal E)\le\left(1-\frac{t^{r-1}}{\binom{rn-1}{r-1}} \right) \left(1-\frac{ (t-r+1)^{r-1}}{\binom{rn-r-1}{r-1}}\right) \ldots
		\left(1-\frac{ [t-(t'-1)(r-1)]^{r-1}}{\binom{rn-(t'-1)r-1}{r-1}}\right).
	$$
	We bound each factor from above, taking the smallest numerator and the largest denominator, by
	\[
	1 - \frac{(t/2)^{r-1}}{\binom{rn-1}{r-1}}
	\le 1 - \frac{(t/2)^{r-1}}{(rn)^{r-1}}
	= 1 - \left(\frac{t}{2rn} \right)^{r-1}
	\le \exp\left\{{- \left(\frac{t}{2rn} \right)^{r-1}}\right\}.
	\]
	Thus,
	\[
	\PP(\mathcal E)\le \exp\left\{{- t' \left(\frac{t}{2rn} \right)^{r-1}}\right\}
	\le \exp\left\{{-\frac{1}{(2r)^r} \cdot \frac{t^r}{n^{r-1}}}\right\}.
	\]
\end{proof}

\begin{rem}
In Appendix~\ref{lemma:span_prob:2} we provide an alternative proof of the above lemma (with a worse constant), based on Chernoff's bound for hypergeometric distribution, in which the online scheme is replaced by the permutational scheme.
\end{rem}


Another ingredient of the proof of the lower bound in Theorem~\ref{thm:random} is a Talagrand's concentration inequality for random permutations from \cite{Talagrand}. We quote here a slightly simplified version from \cite{LuczakMcDiarmid} (see also \cite{McDiarmid}). Let ${\Pi}_N$ be a random permutation of order $N$.

\begin{theorem}[Luczak and McDiarmid \cite{LuczakMcDiarmid}]\label{tala}
	Let $h(\pi)$ be a function defined on the set of all permutations of order $N$ which, for some positive constants $c$ and $d$, satisfies
	\begin{enumerate}[label=\rmlabel]
		\item\label{thm:talagrand:i} if $\pi_2$ is obtained from $\pi_1$ by swapping two elements, then $|h(\pi_1)-h(\pi_2)|\le c$;
		\item\label{thm:talagrand:ii} for each $\pi$ and $s>0$, if $h(\pi)=s$, then in order to show that $h(\pi)\ge s$, one needs to specify only at most $ds$ values $\pi(i)$.
	\end{enumerate}
	Then, for every $\eps>0$,
	$$\PP(|h({\Pi}_N)-m|\ge \eps m)\le4 \exp(-\eps^2 m/(32dc^2)),$$
	where $m$ is the median of $h({\Pi}_N)$.
\end{theorem}
Notice that we can apply this lemma to  random $r$-matchings, as they can be generated by random permutations (c.f.\ the permutational scheme at the beginning of this section).


Finally, we are able to prove Theorem~\ref{thm:random}.

\begin{proof}[Proof of Theorem~\ref{thm:random}]
Let $P$ be a collectable $r$-pattern. In view of Lemma~\ref{lem:upbr}, it only remains to prove  that a.a.s.\ the random matching $\rm^{(r)}_{n}$ contains a $P$-clique of size at least $\Omega_r(n^{1/r})$.
	Let $t=n^{1-\frac{1}{r}}$. For simplicity,  assume that both, $t$ and $k:=n/t=n^{1/r}$ are integers.  We divide $[rn]$ into $rk$ consecutive blocks $B_1,\dots,B_{rk}$, each of length $t$. So,  $B_1=[t]$,  $B_2=\{t+1,\dots,2t\}$, etc.

Let $K$ be a (template) $P$-clique of size $k$ on vertex set $\{v_1,\dots,v_{rk}\}$ disjoint from $[rn]$ and with edges $e_1,\dots,e_k$.
	We may think of the sets $B_j$ as $t$-blow-ups of the vertices $v_j$ of $K$.
	For every $i=1,\dots,k$, let $I_i$ be the indicator random variable equal to 1 if there is an edge in $\rm^{(r)}_{n}$ spanning the sets $B_j$ for all $j$ such that $v_j\in e_i$ and 0 otherwise.
 Further, define $X = \sum_{i=1}^k I_i$.

 Observe that if $X=k$, we would find a copy of $K$ in $\rm^{(r)}_{n}$. More generally (and realistically), since every subset of edges of $K$ forms itself a $P$-clique,  there will be a $P$-clique of size $X$ in $\rm^{(r)}_{n}$ (see Fig.\ \ref{K}). To finish the proof, we are going to show that a.a.s.~$X=\Omega_r(k)$. 
	
\begin{figure}[ht]
\captionsetup[subfigure]{labelformat=empty}
\begin{center}

\scalebox{1}
{
\centering
\begin{tikzpicture}
[line width = .5pt,
vtx/.style={circle,draw,black,very thick,fill=black, line width = 1pt, inner sep=0pt},
]

    \node[vtx] (1) at (0,0) {};
    \node[vtx] (2) at (1,0) {};
    \node[vtx] (3) at (2,0) {};
    \node[vtx] (4) at (3,0) {};
    \node[vtx] (5) at (4,0) {};
    \node[vtx] (6) at (5,0) {};
    \node[vtx] (7) at (6,0) {};
    \node[vtx] (8) at (7,0) {};
    \node[vtx] (9) at (8,0) {};
    \node[vtx] (10) at (8.8,0) {};

    \draw[line width=2.5mm, color=gray, line cap=round]  (1) -- (0.8,0) node[pos=0.5, below] {$B_1$};
    \draw[line width=2.5mm, color=gray, line cap=round]  (2) -- (1.8,0) node[pos=0.5, below] {$B_2$};
    \draw[line width=2.5mm, color=gray, line cap=round]  (3) -- (2.8,0) node[pos=0.5, below] {$B_3$};
    \draw[line width=2.5mm, color=gray, line cap=round]  (4) -- (3.8,0) node[pos=0.5, below] {$B_4$};
    \draw[line width=2.5mm, color=gray, line cap=round]  (5) -- (4.8,0) node[pos=0.5, below] {$B_5$};
    \draw[line width=2.5mm, color=gray, line cap=round]  (6) -- (5.8,0) node[pos=0.5, below] {$B_6$};
    \draw[line width=2.5mm, color=gray, line cap=round]  (7) -- (6.8,0) node[pos=0.5, below] {$B_7$};
    \draw[line width=2.5mm, color=gray, line cap=round]  (8) -- (7.8,0) node[pos=0.5, below] {$B_8$};
    \draw[line width=2.5mm, color=gray, line cap=round]  (9) -- (8.8,0) node[pos=0.5, below] {$B_9$};

    \draw[line width=0.3mm, color=lightgray]  (1) -- (10);
    \fill[fill=black, outer sep=-0.2mm]   (1) circle (0.1) node [below left] {$1$};
    \fill[fill=black, outer sep=-0.2mm]   (10) circle (0.1) node [below right] {$3n$};

    \node[vtx] (v1) at (0.75,0) {};
    \node[vtx] (v3) at (2.5,0) {};
    \node[vtx] (v4) at (3.25,0) {};
    \node[vtx] (v6) at (5.8,0) {};
    \node[vtx] (v7) at (6.4,0) {};
    \node[vtx] (v9) at (8.3,0) {};
    \coordinate (v47) at (4.825,1.25);
    \coordinate (v36) at (4.15,1);
    \coordinate (v69) at (7.05,1);

    \node[color=black] at (-1.5,0) {$\rm^{(3)}_{n}\!\!:$};

    \draw[line width=0.5mm, color=blue, outer sep=2mm] (v4) arc (0:180:1.25);
    \draw[line width=0.5mm, color=blue, outer sep=2mm] plot [smooth, tension=2] coordinates {(v4) (v47) (v7)};
    \draw[line width=0.5mm, color=blue, outer sep=2mm] plot [smooth, tension=2] coordinates {(v3) (v36) (v6)};
    \draw[line width=0.5mm, color=blue, outer sep=2mm] plot [smooth, tension=2] coordinates {(v6) (v69) (v9)};

    \fill[fill=black, outer sep=1mm]   (v1) circle (0.1);
    \fill[fill=black, outer sep=1mm]   (v3) circle (0.1);
    \fill[fill=black, outer sep=1mm]   (v4) circle (0.1);
    \fill[fill=black, outer sep=1mm]   (v6) circle (0.1);
    \fill[fill=black, outer sep=1mm]   (v7) circle (0.1);
    \fill[fill=black, outer sep=1mm]   (v9) circle (0.1);

    \node[vtx] (k1) at (0.5,2.5) {};
    \node[vtx] (k2) at (1.5,2.5) {};
    \node[vtx] (k3) at (2.5,2.5) {};
    \node[vtx] (k4) at (3.5,2.5) {};
    \node[vtx] (k5) at (4.5,2.5) {};
    \node[vtx] (k6) at (5.5,2.5) {};
    \node[vtx] (k7) at (6.5,2.5) {};
    \node[vtx] (k8) at (7.5,2.5) {};
    \node[vtx] (k9) at (8.5,2.5) {};
    \coordinate (k34) at (3,3.75) {};
    \coordinate (k67) at (6,3.75) {};
    \coordinate (k45) at (4,3.5) {};
    \coordinate (k78) at (7,3.5) {};

    \draw[line width=0.3mm, color=lightgray]  (k1) -- (k9);

    \draw[line width=0.5mm, color=blue, outer sep=2mm] (k4) arc (0:180:1.5);
    \draw[line width=0.5mm, color=blue, outer sep=2mm] (k7) arc (0:180:1.5);
    \draw[line width=0.5mm, color=blue, outer sep=2mm] plot [smooth, tension=2] coordinates {(k2) (k34) (k5)};
    \draw[line width=0.5mm, color=blue, outer sep=2mm] plot [smooth, tension=2] coordinates {(k5) (k67) (k8)};
    \draw[line width=0.5mm, color=blue, outer sep=2mm] plot [smooth, tension=2] coordinates {(k3) (k45) (k6)};
    \draw[line width=0.5mm, color=blue, outer sep=2mm] plot [smooth, tension=2] coordinates {(k6) (k78) (k9)};

   \fill[fill=black, outer sep=1mm]   (k1) circle (0.1) node [below] {$v_1$};
   \fill[fill=black, outer sep=1mm]   (k2) circle (0.1) node [below] {$v_2$};
   \fill[fill=black, outer sep=1mm]   (k3) circle (0.1) node [below] {$v_3$};
   \fill[fill=black, outer sep=1mm]   (k4) circle (0.1) node [below] {$v_4$};
   \fill[fill=black, outer sep=1mm]   (k5) circle (0.1) node [below] {$v_5$};
   \fill[fill=black, outer sep=1mm]   (k6) circle (0.1) node [below] {$v_6$};
   \fill[fill=black, outer sep=1mm]   (k7) circle (0.1) node [below] {$v_7$};
   \fill[fill=black, outer sep=1mm]   (k8) circle (0.1) node [below] {$v_8$};
   \fill[fill=black, outer sep=1mm]   (k9) circle (0.1) node [below] {$v_9$};

    \node[color=black] at (-1.15,2.5) {$K\!\!:$};

\end{tikzpicture}
}

\end{center}

\caption{A $P_9$-clique $K$ of size 3 and a copy of its sub-clique in $\rm^{(3)}_{n}$.}
\label{K}
			
\end{figure}
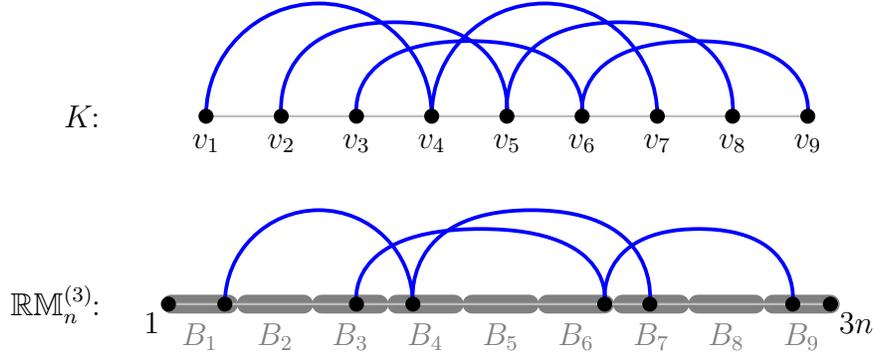

	To this end, observe that by Lemma~\ref{lemma:span_prob} applied separately for each $i$ to the sets $B_j$, $v_j\in e_i$,	
	\[
	\E X = \sum_{i=1}^k \PP(I_i=1)
	\ge \left(1-\exp\left\{-\frac{1}{(2r)^r} \cdot \frac{t^r}{n^{r-1}}\right\}\right)k = \left(1-\exp\left\{-\frac{1}{(2r)^r} \right\}\right)k=:c_rk.
	\]

	Now we just need to show a sharp concentration of $X$ around $\E X$. For this, recalling the permutation scheme of generating $\rm^{(r)}_{n}$,  we view $X$ as a function of $\Pi_{rn}$ and are going to apply Theorem~\ref{tala} with $h({\Pi}_{rn})=X$. We begin with checking its assumptions.

 Observe that swapping two elements of $\Pi_{rn}$ affects at most two edges of $\rm^{(r)}_{n}$, and so it changes   the values of  at most two indicators $I_i$, that is, it changes the value of $X = \sum_{i=1}^k I_i$ by at most 2. Moreover, to exhibit the event $X\ge s$, it is sufficient to reveal, for each $1\le i\le s$, one edge spanning the sets $B_j$, $v_j\in e_i$, which boils down to specifying just $rs$ values of $\Pi_{rn}$.
  Thus, we are in position to apply Theorem~\ref{tala} to $X$ with $N=rn$, $c=2$ and $d=r$.
  Let $m$ be the median of $X$. Then, by Theorem \ref{tala},
	\[
	\PP(|X-m|\ge m/2)\le4\exp(-m/(512r)).
	\]
	Moreover, there is a standard way to switch from the median to the expectation $\mu=\E X$. Indeed, we have (see for example \cite{Talagrand} or Lemma 4.6 in~\cite{McDiarmid1998})  that $|m -\mu|=O(\sqrt m)$. As $\mu\ge c_rk\to\infty$, it follows that $m\to\infty$ and, in particular,  $|m-\mu|\le 0.01\mu$. This implies that $\PP(|X-m|\ge m/2) = o(1)$ and
	\begin{align*}
		\PP(|X-\mu|\ge(3/4)\mu)
		&= \PP(|X-m + m-\mu|\ge(3/4)\mu)\\
		&\le \PP(|X-m| + |m-\mu|\ge(3/4)\mu)\\
		&\le \PP(|X-m| \ge(2/3)\mu)
		\le \PP(|X-m|\ge m/2) = o(1).
	\end{align*}
	Hence, a.a.s.~$X \ge (1/4)\mu\ge (c_r/4)k=(c_r/4)n^{1/r}$, finishing the proof.
\end{proof}

%
%
%
%
\begin{rem}
If we only aimed at showing the a.a.s.\ presence of $P$-cliques of size at least $n^{1/r}/\omega(n)$, where $\omega:=\omega(n)\to\infty$ arbitrarily slowly, then we would do without Talagrand's inequality. Indeed, setting $t=\omega n^{1-1/r}$ and $Y=k-X$, by Markov's inequality and Lemma~\ref{lemma:span_prob},
$$\Pr(X\le k/2)=\Pr(Y\ge k/2)\le2\E Y/k=2\Pr(I_1=0)\le\exp\left\{-\frac{\omega^r}{(2r)^r} \right\}=o(1).$$

\end{rem}

\begin{rem}
  Recall that the $r$-pattern $P_1=A^rB^r$ is called an alignment and  a $P_1$-clique -- a line. So, a line of size $k$ is just a matching represented by the word $A_1^rA_2^r\cdots A_k^r$ consisting of $k$ consecutive runs of $k$ different letters, each run of length $r$. It turns out that for lines the proof of the lower bound in Theorem~\ref{thm:random} is  more elementary than for other $P$-cliques (as it does not use Talagrand's inequality). We state and prove this special case of Theorem~\ref{thm:random} in Appendix~\ref{thm:lines}.
\end{rem}

\begin{rem}\label{avoidTal}
In fact, as outlined by a referee, there is a proof of the full version of Theorem~\ref{thm:random} which avoids  Talagrand's inequality. It is still based on the permutational scheme and, in addition, on Chernoff's bound for hypergeometric distribution and, ultimately, on Chebyshev's inequality. It is quite technical, though, and we decided  not to present it here (see Remark~\ref{B1} in Appendix B).
\end{rem}

\section{Final remarks}

\subsection{Optimality}\label{opt}
Here we discuss the issue of optimality of Theorem \ref{Theorem ESz General_classic}. The case $r=2$, that is, Theorem \ref{Theorem E-S for LSW}, is best possible, as demonstrated by an explicit construction in \cite{DGR-match}. However, already for $r=3$, we can show the optimality of Theorem \ref{Theorem ESz General_classic} only in special cases  when some of the parameters $a_i$ are set to 1 (cf.\  \cite[Remark 7]{DGR-match}).

To describe the constructions, we define a kind of ``blow-up'' operation on pairs of $r$-matchings. Given $r$-matchings $M$ and $N$, the latter of size $t$, the \emph{$N$-blow-up}  of $M$ is the matching $M[N]$ obtained from $M$ by replacing every vertex $i$ of $M$ by an ordered set $U_i$, $|U_i|=t$, and each edge $e=\{i_1,\dots,i_r\}\in M$ by a copy $N_e$ of $N$ on $\bigcup_{j=1}^rU_{i_j}$ (see Fig. \ref{NM}). Note that $|M[N]|=|M||N|$.

\begin{figure}[ht]
\captionsetup[subfigure]{labelformat=empty}
\begin{center}

$$M=ABACCBABC\qquad\mbox{and}\qquad N=XYZ\;YZX\;ZXY$$

$$M[N]=\underbrace{DEF}_A\;\underbrace{GHI}_B\;\underbrace{EFD}_A\;\underbrace{JKL}_C\;\underbrace{KLJ}_C\;\underbrace{IHG}_B\;\underbrace{FDE}_A\;
\underbrace{IGH}_B\;\underbrace{LJK}_C$$

\end{center}

\caption{An example of an $N$-blow-up of $M$.}
\label{NM}
		
\end{figure}

In $M[N]$ there are two kinds of pairs of edges: 1) both within the same copy of $N$ (like $e_D$ and $e_F$ in Fig. \ref{NM}), call them \emph{$N$-pairs}, and 2) each from a different copy of $N$ (like $e_D$ and $e_G$ in Fig. \ref{NM}), call them \emph{$M$-pairs}. Obviously, every $N$-pair forms the same pattern as the corresponding pair of edges in $N$ (e.g., edges $e_D$ and $e_F$ in Fig. \ref{NM} form in $M[N]$ the same pattern as edges $e_X$ and $e_Z$ in $N$). Let us call this property \emph{$N$-inheritance}.

But this operation is particularly useful
when,
in addition to $N$-inheritance,  every $M$-pair made of one edge from $N_e$ and one from $N_f$, where $e,f\in M$, forms the same pattern as the  pair $\{e,f\}$ does in $M$ (e.g., edges $e_D$ and $e_G$ in Fig. \ref{NM} form in $M[N]$ the same pattern as edges $e_A$ and $e_B$ in $M$). Let us call this property \emph{$M$-inheritance}. 

For a matching $M$, let $\mathcal P(M)$ be the set of all patterns appearing among the pairs of edges of $M$. For instance, in Fig. \ref{NM}, we have $\mathcal P(M)=\{P_3,P_5,P_9\}$ (cf. Table~\ref{table:relations}). A crucial observation is that if $M[N]$ is $M$-inheritable and $\mathcal P(M)\cap\mathcal P(N)=\emptyset$, then the size of a largest $P$-clique in $M[N]$ is, for every $P\in\mathcal P(M)$, equal to the size of a largest $P$-clique in $M$ and the same is true for every $P\in\mathcal P(N)$.

In particular, $M[N]$ is $M$-inheritable if $N$ is $r$-partite. Consequently, if $M$ is  $r$-partite too, then so is $M[N]$. Also, note that if $M$ is a line, that is, an $A^rB^r$-clique of size $\ell$, then $M[N]$ is an ordered concatenation of $\ell$  copies of $N$, and so it is trivially $M$-inheritable regardless of the structure of $N$.

 Now, we are ready to describe our constructions, beginning with $r=2$ and $r=3$. Recall that for $r=2$, both $R_2=ABBA$ and $R_3=ABAB$ are bipartite, and so every $R_2$-clique and $R_3$-clique, is bipartite, while for $r=3$, only patterns $P_5,P_6,P_8,P_9$ are tripartite, and so are $P_i$-cliques for $i=5,6,8,9$.

 For $r=2$ the optimal construction  given in \cite{DGR-match}  is just $L_\ell[S_s[W_w]]$, that is, a concatenation of $\ell$ copies of the $W_w$-blow-up of $S_s$, where $S_s$ is a stack, or $R_2$-clique, of size $s$, and $W_w$ is a wave, or $R_3$-clique, of size $w$. Indeed, by the observations we have just made, the largest line in $L_\ell[S_s[W_w]]$ has size $\ell$, the largest stack - size $s$, and the largest wave --size $w$, while the size of $L_\ell[S_s[W_w]]$ is precisely $\ell s w$.
   By swapping the two inner arguments in the blow-up operation, one can obtain another optimal construction,  $L_\ell[W_w[S_s]]$. Also,  legitimate constructions are $(L_\ell[S_s])[W_w]]$ and $(L_\ell[W_w])[S_s]$ where we begin by blowing up the line.

 For $r=3$, let $K_i$ be a $P_i$-clique of size $a_i$, $i=2,5,6,8,9$. Consider first the 3-fold blow-up operation $K_5[K_6[K_8[K_9]]]$. It has size $a_5a_6a_8a_9$ and, by the above observations the largest size of a $P_i$-clique, $i=5,6,8,9$, is $a_i$. This shows that both, Corollary~\ref{Theorem ESz General_partite} and Corollary~\ref{Theorem ESz General_diagonal}(c) are optimal for $r=3$. In fact, this is true for every $r\ge2$, as there is a straightforward generalization of the above construction where one superimposes $2^{r-1}-1$ blow-up operations (involving all $2^{r-1}$ $r$-partite patterns in any order whatsoever.)

  To get beyond the 3-partite patterns, consider the 5-fold blow-up construction \newline $K_1[K_2[K_5[K_6[K_8[K_9]]]]]$. Note that it has $a_1a_2a_5a_6a_8a_9$ edges. Note also that $K_5[K_6[K_8[K_9]]]$ is still 3-partite, so  $K_2[K_5[K_6[K_8[K_9]]]]$ is $K_2$-inheritable, while the final construction is $K_1$-inheritable due to the robustness of lines. It follows that every pair of edges  forms one of the six patterns and the largest $P_i$-clique  is still of size $a_i$, $i=1,2,5,6,8,9$.
This shows that Theorem \ref{Theorem ESz General_classic}(b) is optimal if,  $a_3=a_4=a_7=1$, and more generally, if some three of parameters $a_2,a_3,a_4,a_7$ are set to 1.

For larger $r$, in  similar constructions, the sets of non-marginal values of parameters $a_i$ are getting proportionally smaller.
Indeed, there are exactly $2^{r-1}$ $r$-partite $r$-patterns, so all but $2^{r-1}+2$ parameters $a_i$ must be set to~1. This implies that there is an $r$-matching of size $n$ with a largest $P$-clique (for any $P$) of size $O(n^{1/(2^{r-1}+2)})$. For $r=3$, this is $O(n^{1/6})$, while the lower bound in Corollary~\ref{Theorem ESz General_diagonal}(a) is  $\Omega(n^{1/9})$.

\begin{rem}\label{refinement} A referee pointed out that for both $M$ and $N$ clean, a sufficient condition for $M$-inheritance of the blow-up $M[N]$ is that every
pattern in $N$  yields an ordered partition of $r$ (as defined in the proof of Corollary~\ref{number_coll}) which refines every ordered partition associated with a pattern in $M$. (So, $r$-partite patterns with partition $r=1+\cdots+1$ and lines with partition $r=r$ are the two extremes.) This allows to expand the blow-ups and get better constructions. 

For $r=3$, one could add to the string of blow-ups $K_3$, as both $P_2$ and $P_3$ yield partition $3=2+1$, obtaining a clean matching of size $a_1a_2a_3a_5a_6a_8a_9$ with the largest $P_i$-clique of size $a_i$ (assuming $a_4=a_7=1$).
Alternatively,  one could replace  the pair $P_2,P_3$ with $P_4,P_7$ (both latter patterns have partition $3=1+2$). Either way, in the symmetric case this yields an upper bound of $O(n^{1/7})$.   

For general $r$, as each  partition of $r$ with $i$ terms corresponds to $2^{i-1}$ patterns, one can fix a maximum chain of refining partitions and obtain an iterated blow-up of $\sum_{i=1}^r2^{i-1}=2^r-1$ $P$-cliques. For instance, for $r=4$, one could take the chain of partitions $4,\;3+1,\;2+1+1,\;1+1+1+1$ and the corresponding 14-fold blow-up (cf. Table~\ref{table:patterns4})
$$K_1[K_2[K_3[K_5[K_6[K_8[K_9[K_{14}[K_{15}[K_{17}[K_{18}[K_{23}[K_{24}[K_{26}[K_{27}]]]]]]]]]]]]]]]$$
where $K_i$ is a $P_i$-clique of size $a_i$. (Note that the rightmost eight cliques correspond to the eight 4-partite patterns, all with partition $1+1+1+1$.)
\end{rem}

\medskip

There are other than blow-ups constructions leading to optimal cases.
 One such case is when $r=3$, $a_1=n/2$, $a_3=2$ (or $a_7=2$),  and all other seven parameters $a_i$ are set to~1. It is realized  by a simple chain-like construction.
For an $r$-pattern $P$, define a \emph{$P$-chain} as a matching in which consecutive  edges (in the order of the left ends) form pattern $P$, while all other pairs of edges form the alignment $P_1=A^rB^r$. For $r=3$, one can easily construct a $P_i$-chain for $i=3,7$, which shows  optimality of Theorem \ref{Theorem ESz General_classic}(b) in the above mentioned case.
For example,  the $P_3$-chain looks like this:
\[
A_1A_1A_2A_2A_1A_3A_3A_2A_4A_4A_3A_5A_5A_4A_6A_6A_5A_7A_7\dots A_{n-2}A_nA_nA_{n-1}A_n.
\]

For general $r$, the same construction works for all patterns $P$ in which the first $A$-run is at least as long as the number of $B$'s preceding the last $A$, which equals $r$ minus the length of the $B$-run at the end of $P$. Among collectable patterns, for $r=3$ this property is satisfied only by $P_3$ and $P_7$, while for $r=4$, only by $P_3$, $P_7$, and $P_{19}$ (cf.\ Table \ref{table:patterns4}).

But $P$-chains with non-collectable $P$ can also be useful in showing optimality of Theorem~\ref{Theorem ESz General_classic}(a) for some very special cases.
Indeed, for $r=3$, consider the $P_1^*$-chain
\[
A_1A_1A_2\;A_1A_2A_3\;A_2A_3A_4\;A_3A_4A_5\;... \;A_{n-3}A_{n-2}A_{n-1}\;A_{n-2}A_{n-1}A_n\; A_{n-1}A_nA_n
\]
(the spaces between 3-letter blocks where introduced solely for the benefit of the reader) and observe that the consecutive edges form pattern $P^*_1$, while all other pairs of edges form pattern $P_1$, that is, they form a $P_1$-clique of size $\lceil n/2\rceil$.  This shows that in the special case when $a_1=n$ is even (with  $a_i=1$ for all other $i$), the factor of $1/2$ appearing in the assertion of Theorem \ref{Theorem ESz General_classic}(a)   cannot be improved.

\medskip

On the negative side, let us point to instances which are far from optimality. Consider the case when $r=3$ and all $a_i=1$, except $i=2$ and $i=4$. Let every pair of edges of a 3-matching $M\in\cM_n^{(3)}$ form either pattern $P_2=AABBBA$ or $P_4=ABBBAA$. Let the edges of $M$ be (in the order of their left ends) $e_1,\dots,e_n$ with the corresponding letters $A_1,\dots,A_n$. W.l.o.g., let $e_1$ and $e_2$ form pattern $P_2$, i.e., there is a subsequence $A_1A_1A_2A_2A_2A_1$. Then also $e_1$ and $e_j$, $j=3,\dots,n$, must form pattern $P_2$, because no matter what the relation between $e_2$ and $e_j$ is, $e_j$ is totally ``inside'' $e_2$. For the same token, if $e_1$ and $e_2$ form pattern $P_4$, then so do $e_1$ and $e_j$, $j=3,\dots,n$. And the same is true for every $e_h$, $h\ge1$, that is, for every $j\ge h$, $e_h$ and $e_j$ form the same pattern, always $P_2$ or always $P_4$. This partitions the edges of $M\setminus\{e_n\}$ into two subsets, $M_2$ and $M_4$ such that $M_2\cup\{e_n\}$ is a $P_2$-clique and $M_4\cup\{e_n\}$ is a $P_4$-clique. It follows that for every partition $n-1=n_2+n_4-3$, there is a $P_2$-clique of size $n_2$ or a $P_4$-clique of size  $n_4$, a much stronger statement than what is implied by Theorem \ref{Theorem ESz General_classic}. And the same is true for pairs $(P_2,P_7)$ and $(P_3,P_4)$.

Even more surprising is the combination of $P_3$ and $P_7$ ($r=3$). One can easily check that when the edges of $M\in\cM_n^{(3)}$ form only patterns $P_3$ and $P_7$, then, in fact, $M$ must be either entirely a $P_3$-clique or entirely a $P_7$-clique. Indeed, if edges $e_A$ and $e_B$ form  pattern $P_7=ABAABB$ and $e_B$ and $e_C$ form pattern $P_3=BBCCBC$, then $e_A$ and $e_C$ are forced to form pattern $P_1$, a contradiction. Hence, we have $BCBBCC$ and, in turn, $e_A$ and $e_C$ also form $P_7$. If, on the other hand,  we have $AABBAB$ and $BCBBCC$, then $AACACC=P_1^*$, a contradiction again. So, again, all three pairs of edges among $e_A,e_B,e_C$ form pattern $P_3$. This is the most striking case against the optimality of Theorem \ref{Theorem ESz General_classic}.

\subsection{Improvement for $r=3$}\label{impr} We have recently learned a cute argument which,  for $r=3$, improves Corollary~\ref{Theorem ESz General_diagonal}(a) from $\Omega(n^{1/9})$ to $\Omega(n^{1/7})$, matching the upper bound $O(n^{1/7})$ from Remark~\ref{refinement}. For a set of 3-patterns $\cP$, a \emph{$\cP$-clique} is a matching whose all pairs of edges form a pattern from $\cP$. Let $M$ be a 3-matching of size $n\ge a_1(a_2+a_4)a_6a_8(a_3+a_7)a_5a_9+1$. The claim is that for some $i\in[9]$, $M$ contains a $P_i$-clique of size $a_i+1$.
The starting point is, unlike in our inductive proof, to drop the middle vertex of each edge of $M$ and apply to the 2-matching $M'$ obtained this way Theorem~\ref{Theorem E-S for LSW} with $\ell=a_1$, $s=(a_2+a_4)a_6a_8$, and $w=(a_3+a_7)a_5a_9$. It is easy to check that when two edges of $M'$ form a nesting, then the corresponding edges of $M$ form one of the patterns in $\cP=\{P_2,P_4,P_6,P_8\}$, while if they form a crossing, then the corresponding edges of $M$ form one of the patterns in $\cP'=\{P_1^*,P_3,P_5,P_7,P_9\}$. We thus obtain in $M$ either a line of size $a_1+1$ or a $\cP$-clique $H$ of size $h:=(a_2+a_4)a_6a_8+1$, or a $\cP'$-clique $H'$ of size $h':=(a_3+a_7)a_5a_9+1$. Now our goal is to somehow separate patterns $P_2$ and $P_4$ in~$H$, and patterns $P_3$ and $P_7$ in $H'$.

To this end, let us say that a vertex $v$ in a 3-matching is \emph{covered} by an edge $e$ of that matching if $v$ is between the first and the last vertex of $e$ ($v$ does not need to be the middle vertex of $e$). Let $e_1,\dots, e_h$ be the edges of $H$ ordered by their left ends. Notice, and this is absolutely crucial, that the set of edges of $H$ which cover a fixed vertex $v$ of $H$ form a prefix of the sequence $e_1,\dots, e_h$ (possibly empty). This is because each pattern in $\cP$ begins and ends with the same letter. So, if an edge $e_j$ covers a vertex, then so does each $e_i$ for $i\le j$.

If $v$ itself is a midpoint of an edge $e_j$ and the edges $e_g,\dots, e_h$ do \emph{not} cover $v$, then $v$ is either entirely to the left of all these edges or entirely to the right.  This brings about a classification of all edges of $H$ into two categories: the \emph{left edges} -- those whose midpoints are to the left of all edges not covering it,  and the \emph{right edges}. (If all edges of $H$ cover $v$, then $e_j$ can be classified either way.)
Here is an example (note that the extreme points of edges form mutually a nesting):
$$A\;A\;B\;C\;D\;C\;D\;E\;F\;F\;B\;F\;E\;E\;D\;C\;B\;A.$$
The midpoint of $e_C$ is covered by edges $e_A,e_B,e_C,e_D$ and is to the left of edges $e_E,e_F$, while the midpoint of $e_E$ is not covered only by $e_F$ and lies to the right of $e_F$. In this example, edges $e_A,e_C,e_D$ are left, $e_E$ is right, while $e_B,e_F$ can be both (their midpoints are covered by all six edges).

 Another crucial observation is that no left edges form pattern $P_4$, while no right edges form pattern $P_2$. Indeed, if $e_i=XXX$ and $e_j=YYY$, $i<j$, form pattern $P_4=XYYYXX$, then the midpoint of $e_i$ is not covered by $e_j$, while it is to the right of $e_j$. Similarly, if $e_i$ and $e_j$ form $P_2=XXYYYX$, then the midpoint of $e_i$ (the second $X$) is to the left of $e_j$ (and $e_j$ does not cover it).

  So, if  $h\ge a_2a_6a_8+1$, then, by Theorem~\ref{Theorem ESz General_classic}(b) applied to the set of left edges of $H$, for some $i\in\{2,6,8\}$ there is a $P_i$-clique in $H$, and thus in $M$, of size $a_i+1$. On the other hand, if  $h\ge a_4a_6a_8+1$, then,  Theorem~\ref{Theorem ESz General_classic}(b) applied to the set of right edges of $H$, for some $i\in\{4,6,8\}$ there is a $P_i$-clique in $H$ of size $a_i+1$. And one of the two statements must hold, since $h:=(a_2+a_4)a_6a_8+1$.

The situation is similar, though a bit more complicated, for $H'$.  Here the edges covering a fixed vertex may form a prefix or a suffix of the sequence of edges of $H'$ ordered by their left ends. This also yields a binary classification of the edges of $H'$ with a similar mutual exclusion of patterns $P_3$ and $P_7$, and of pattern $P_1^*$ whatsoever. We leave the details for the reader. In conclusion, if $n\ge a_1(a_2+a_4)a_6a_8(a_3+a_7)a_5a_9+1$, then for some $i\in[9]$, $M$ contains a $P_i$-clique of size $a_i+1$, which, setting all $a_i$ equal, yields the lower bound $\Omega(n^{1/7})$.

But in general, this statement
is still far from optimal. If all $a_i=1$ except for $a_2$ and $a_7$, then, according to the discussion in the previous subsection, the correct bound should be
$\Theta(a_2 + a_7)$, not $\Theta(a_2a_7)$ given above. A similar situation holds for $a_3$ and $a_4$.

\subsection{Open questions} Let us conclude the paper with some questions concerning possible future research. Firstly, in view of the discussion in the previous two subsections, one could try to strengthen Theorem \ref{Theorem ESz General_classic} so that the new version would be indeed optimal.
\begin{prob}
 Given $r\ge3$, let $P_1,\dots, P_{3^{r-1}}$ be all collectable ordered $r$-matchings. Find a function $f(a_1,\dots,a_{3^{r-1}})$ such that every $r$-matching of size $n\ge f(a_1,\dots,a_{3^{r-1}})+1$ contains for some $i$ a $P_i$-clique of size $a_i+1$ and, on the other hand, there exists an $r$-matching with $f(a_1,\dots,a_{3^{r-1}})$ edges not containing a $P_i$-clique of size $a_i+1$ for any $i$.
 \end{prob}

Another direction in which one could sail inspired by our results is a \emph{cyclic} counterpart of ordered matchings.
 A cyclic version of the original Erd\H{o}s-Szekeres theorem for permutations was obtained recently by Czabarka and Wang \cite{Czabarka}.
 A \emph{cyclic $r$-matching} of size $n$ is just an $r$-uniform matching with $n$ edges on a cyclically ordered set of vertices. Such matchings can be naturally represented by cyclic words in which every letter corresponding to an edge occurs exactly $r$ times. 
 For instance, for $r=2$ we have only two possible cyclic patterns, namely $Q_1=AABB$ and $Q_2=ABAB$ (since $ABBA$ is equivalent to $AABB$), while for $r=3$ there are just four of them: $AAABBB$, $AABABB$, $AABBAB$, and $ABABAB$.
 
It would be desirable to know whether in this setting a similar Erd\H os-Szekeres type phenomena occur. First, however, we should decide what is the appropriate analog of a $P$-clique. Take a look, for instance, at two cyclic matchings of size three, $AABBCC$ and $AABCCB$. They both form a $Q_1$-clique in a sense that each pair of edges form a cyclic pattern $Q_1$, but they are not isomorphic. Perhaps, to remedy this problem,  the right  definition of a homogenous cyclic $r$-matching of size $n$ should demand that, for each $k\le n$, all sub-matchings of size $k$ are mutually order-isomorphic.

\begin{prob} 
For each $r\ge 2$ define homogenous cyclic matchings ``in a right way''. Then
identify and/or characterize all possible homogenous cyclic ordered $r$-matchings. Subsequently, prove an Erd\H os-Szekeres type theorem for them.
	\end{prob}

\medskip

Our next  question refers to the random setting  and aims to pinpoint the multiplicative constants in Theorem~\ref{thm:random}.

\begin{conj}
	For every $r\geqslant 2$ and every collectable $r$-pattern $P$, there exists a positive constant $b_P$ such that the maximum size of a $P$-clique in $\rm^{(r)}_n$ is a.a.s.\ equal to $(1+o(1))b_Pn^{1/r}$.
\end{conj}

\begin{prob}
Assuming the conjecture is true, determine all constants $b_P$.
\end{prob}

The conjecture is true (and the problem is solved) for $r=2$. Indeed,   Stanley \cite{Stanley}  deduced  from a deep result of Baik and Rains \cite{BaikRains} concerning monotone subsequences in random permutations that the maximum size of stacks and waves in $\rm^{(2)}_n$ is a.a.s  $(1+o(1))\sqrt{2n}$.  In turn,  Justicz, Scheinerman, and Winkler in \cite{JSW} showed, while studying a related notion of random interval graphs, that
the size of the largest line contained in $\rm_n$ is a.a.s.\ equal to $(2+o(1))\sqrt{n/\pi}$.

\medskip

Let us conclude with  asking for the size of the largest twins in ordered $r$-matchings, that is,  pairs of order-isomorphic, disjoint sub-matchings. Let $t^{(r)}(M)$  denote the maximum size of twins in an $r$-matching $M$ (measured by the size of just one of them) and $t^{(r)}(n)$ -- the minimum of $t^{(r)}(M)$ over all $r$-matchings on $[rn]$. In \cite[Theorem 20]{DGR-match} we proved that a.a.s.\ $t^{(2)}(\rm^{(2)}_n)=\Theta(n^{2/3})$ and, by linking the presence of twins in 2-matchings with the presence of twins in permutations, we  derived an estimate  $t^{(2)}(n)=\Omega(n^{3/5})$. We conjectured in \cite[Conjecture 23]{DGR-match}  (in the more general setting of multiple twins) that $t^{(2)}(n)=\Omega(n^{2/3})$.

In the $r$-uniform scenario, we proved in \cite[Theorem 1]{DGR_ZG}  that a.a.s.\   $t^{(r)}(\rm^{(r)}_n)=\Theta_r(n^{2/(r+1)})$.
However, its deterministic, much harder counterpart remains wide open.

\begin{prob}
Determine the order of magnitude of $t^{(r)}(n)$.
\end{prob}
Currently, we only know that $t^{(r)}(n)=\Omega_r\left(n^{\frac3{5(2^{r-1}-1) }}\right)$ (see \cite[Theorem 2]{DGR_ZG}) and, as a consequence of the above average case result, that $t^{(r)}(n)=O_r\left(n^{2/(r+1)}  \right)$.
Even a similar question for permutations is far from being settled (see, e.g.,~\cite{BukhR}).

\begin{rem} Several months after posting our paper on arXiv, two new papers have emerged, namely \cite{AJKS} and \cite{SZ}, bringing some further developments and improvements of our results.
\end{rem}

\subsection*{Acknowledgments} We are extremely grateful to both referees for their extended and deep comments which have led to a gross improvement of the presentation of our results. In particular, we owe the short proof of Proposition~\ref{Proposition P-family-clique} as well as the argument presented in Subsection~\ref{impr} and Remark~\ref{refinement} to Referee 1 (5-page report), while the alternative proof of Lemma~\ref{lemma:span_prob} and the contents of Remarks~\ref{avoidTal} and~\ref{B1} were suggested by Referee 2 (3-page report).



\appendix
\section{Proof of Theorem \ref{Theorem E-S for LSW}}\label{appendix}

Since the proof of Theorem \ref{Theorem ESz General_classic} is inductive with the base step $r=2$, we provide here, for completeness, a proof of Theorem \ref{Theorem E-S for LSW}. We chose a variant which differs slightly from those given in \cite{DGR-match} and \cite{HJW2019}.

\begin{proof}[Proof of Theorem \ref{Theorem E-S for LSW}] Let $M$ be an ordered matching consisting of edges $\{a_i,b_i\}$, $i=1,2,\ldots, n$, with the left ends satisfying $a_1<\cdots <a_n$. Notice that the right ends of the edges define a permutation $\pi=(j_1,j_2,\ldots,j_n)$ accordingly to the order  $b_{j_1}<b_{j_2}<\cdots<b_{j_n}$.

By the original Erd\H{o}s-Szekeres theorem this permutation contains either a decreasing subsequence of length $s+1$ or an increasing subsequence of length $p=w\ell+1$. In the former case we are done, since any such decreasing subsequence corresponds to a stack. In the latter case we get a sub-matching $L$ with $p$ edges whose right ends come in the same order as the left ends. We call $L$ \emph{a landscape} (see Fig.\ \ref{ESZ3}). Notice that no pair of edges in a landscape may form a nesting.
\begin{figure}[ht]
\captionsetup[subfigure]{labelformat=empty}
\begin{center}

\scalebox{1}
{
\centering
\begin{tikzpicture}
[line width = .5pt,
vtx/.style={circle,draw,black,very thick,fill=black, line width = 1pt, inner sep=0pt},
]

    \coordinate (0) at (0.5,0) {};
    \node[vtx] (1) at (1,0) {};
    \node[vtx] (2) at (2,0) {};
    \node[vtx] (3) at (3,0) {};
    \node[vtx] (4) at (4,0) {};
    \node[vtx] (5) at (5,0) {};
    \coordinate (55) at (5,1.5) {};
    \node[vtx] (6) at (6,0) {};
    \node[vtx] (7) at (7,0) {};
    \coordinate (77) at (7,1.5) {};
    \node[vtx] (8) at (8,0) {};
    \node[vtx] (9) at (9,0) {};
    \coordinate (99) at (9.5,1.5) {};
    \node[vtx] (10) at (10,0) {};
    \coordinate (11) at (10.5,0) {};

    \draw[line width=0.3mm, color=lightgray]  (0) -- (11);
    \fill[fill=black, outer sep=1mm]  (1) circle (0.1) node [below] {\textcolor{blue}{$A$}};
    \fill[fill=black, outer sep=1mm]  (2) circle (0.1) node [below] {\textcolor{blue}{$B$}};
    \fill[fill=black, outer sep=1mm]  (3) circle (0.1) node [below] {\textcolor{blue}{$C$}};
    \fill[fill=black, outer sep=1mm]  (4) circle (0.1) node [below] {\textcolor{red}{$A$}};
    \fill[fill=black, outer sep=1mm]  (5) circle (0.1) node [below] {\textcolor{red}{$B$}};
    \fill[fill=black, outer sep=1mm]  (6) circle (0.1) node [below] {\textcolor{blue}{$D$}};
    \fill[fill=black, outer sep=1mm]  (7) circle (0.1) node [below] {\textcolor{red}{$C$}};
    \fill[fill=black, outer sep=1mm]  (8) circle (0.1) node [below] {\textcolor{red}{$D$}};
    \fill[fill=black, outer sep=1mm]  (9) circle (0.1) node [below] {\textcolor{blue}{$E$}};
    \fill[fill=black, outer sep=1mm]  (10) circle (0.1) node [below] {\textcolor{red}{$E$}};

    \draw[line width=0.5mm, color=black, outer sep=2mm] (4) arc (0:180:1.5);
    \draw[line width=0.5mm, color=black, outer sep=2mm] (5) arc (0:180:1.5);
    \draw[line width=0.5mm, color=black, outer sep=2mm] plot [smooth, tension=2] coordinates {(3) (55) (7)};
    \draw[line width=0.5mm, color=black, outer sep=2mm] plot [smooth, tension=2] coordinates {(6) (77) (8)};
    \draw[line width=0.5mm, color=black, outer sep=2mm] plot [smooth, tension=2] coordinates {(9) (99) (10)};

\end{tikzpicture}
}

\end{center}

\caption{A landscape with five edges.}
\label{ESZ3}
		
\end{figure}
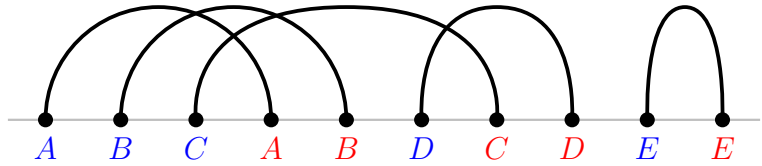

 Let us order the edges of $L$ as $e_1<e_2<\cdots<e_p$, accordingly to the linear order of their left ends. Decompose $L$ into edge-disjoint waves, $W_1, W_2,\ldots, W_k$, in the following greedy way. For the first wave $W_1$, pick $e_1$ and all edges whose left ends are between the two ends of $e_1$, say, $W_1=\{e_1<e_2<\ldots<e_{i_1}\}$, for some $i_1\geqslant 1$. Clearly, $W_1$ is a genuine wave since there are no lines (and no nestings) in $W_1$.
 Also notice that the edges $e_1$ and $e_{i_1+1}$ form an alignment, since otherwise the latter edge would be included in $W_1$.

 Now, we may remove the wave $W_1$ from $L$ and repeat this step for $L-W_1$ to get the next wave $W_2=\{e_{i_1+1}<e_{i_1+2}<\ldots<e_{i_2}\}$, for some $i_2\geqslant i_{1}+1$. We iterate this procedure until there are no edges of $L$ left. Let the last wave be $W_k=\{e_{i_{k-1}+1}<e_{i_{k-1}+2}<\ldots<e_{i_k}\}$, with $i_k\geqslant i_{k-1}+1$. Clearly, the sequence $e_1<e_{i_1+1}<\ldots<e_{i_{k-1}+1}$ of the leftmost edges of the waves $W_i$, $i=1,\dots,k$,  forms a line (see Fig.~\ref{ESZ4}).
\begin{figure}[ht]
\captionsetup[subfigure]{labelformat=empty}
\begin{center}

\scalebox{1}
{
\centering
\begin{tikzpicture}
[line width = .5pt,
vtx/.style={circle,draw,black,very thick,fill=black, line width = 1pt, inner sep=0pt},
]

    \tikzmath{\step = 0.65; \h0 = 1.5; \h1 = 2; \h2=2.5;}
    \foreach \i in {1,...,24}
    {
         \node[vtx] (\i) at (\i*\step,0) {};
         \coordinate (\i00) at (\i*\step,\h0) {};
         \coordinate (\i01) at (\i*\step+\step/2,\h0) {};
         \coordinate (\i10) at (\i*\step,\h1) {};
         \coordinate (\i11) at (\i*\step+\step/2,\h1) {};
         \coordinate (\i20) at (\i*\step,\h2) {};
         \coordinate (\i21) at (\i*\step+\step/2,\h2) {};
    }

    \coordinate (0) at (0.25,0) {};
    \coordinate (25) at (25*\step-0.25,0) {};

    \draw[line width=0.8mm, color=blue, outer sep=2mm] plot [smooth, tension=2] coordinates {(1) (201) (4)};
    \draw[line width=0.5mm, color=blue, outer sep=2mm] plot [smooth, tension=2] coordinates {(2) (500) (8)};
    \draw[line width=0.5mm, color=blue, outer sep=2mm] plot [smooth, tension=2] coordinates {(3) (700) (11)};

    \draw[line width=0.8mm, color=red, outer sep=2mm, dash pattern=on 10pt off 3pt] plot [smooth, tension=2] coordinates {(5) (811) (12)};
    \draw[line width=0.5mm, color=red, outer sep=2mm, dash pattern=on 10pt off 3pt] plot [smooth, tension=2] coordinates {(6) (1011) (15)};
    \draw[line width=0.5mm, color=red, outer sep=2mm, dash pattern=on 10pt off 3pt] plot [smooth, tension=2] coordinates {(7) (1111) (16)};
    \draw[line width=0.5mm, color=red, outer sep=2mm, dash pattern=on 10pt off 3pt] plot [smooth, tension=2] coordinates {(9) (1311) (18)};
    \draw[line width=0.5mm, color=red, outer sep=2mm, dash pattern=on 10pt off 3pt] plot [smooth, tension=2] coordinates {(10) (1411) (19)};

    \draw[line width=0.8mm, color=ForestGreen, outer sep=2mm, densely dotted] plot [smooth, tension=2] coordinates {(13) (1720) (21)};
    \draw[line width=0.5mm, color=ForestGreen, outer sep=2mm, densely dotted] plot [smooth, tension=2] coordinates {(14) (1820) (22)};
    \draw[line width=0.5mm, color=ForestGreen, outer sep=2mm, densely dotted] plot [smooth, tension=2] coordinates {(17) (2020) (23)};
    \draw[line width=0.5mm, color=ForestGreen, outer sep=2mm, densely dotted] plot [smooth, tension=2] coordinates {(20) (2220) (24)};

    \draw[line width=0.3mm, color=lightgray]  (0) -- (25);
    \fill[fill=black, outer sep=1mm]  (1) circle (0.1) node [below] {\textcolor{blue}{$A$}};
    \fill[fill=black, outer sep=1mm]  (2) circle (0.1) node [below] {\textcolor{blue}{$B$}};
    \fill[fill=black, outer sep=1mm]  (3) circle (0.1) node [below] {\textcolor{blue}{$C$}};
    \fill[fill=black, outer sep=1mm]  (4) circle (0.1) node [below] {\textcolor{blue}{$A$}};
    \fill[fill=black, outer sep=1mm]  (5) circle (0.1) node [below] {\textcolor{red}{$D$}};
    \fill[fill=black, outer sep=1mm]  (6) circle (0.1) node [below] {\textcolor{red}{$E$}};
    \fill[fill=black, outer sep=1mm]  (7) circle (0.1) node [below] {\textcolor{red}{$F$}};
    \fill[fill=black, outer sep=1mm]  (8) circle (0.1) node [below] {\textcolor{blue}{$B$}};
    \fill[fill=black, outer sep=1mm]  (9) circle (0.1) node [below] {\textcolor{red}{$G$}};
    \fill[fill=black, outer sep=1mm]  (10) circle (0.1) node [below] {\textcolor{red}{$H$}};
    \fill[fill=black, outer sep=1mm]  (11) circle (0.1) node [below] {\textcolor{blue}{$C$}};
    \fill[fill=black, outer sep=1mm]  (12) circle (0.1) node [below] {\textcolor{red}{$D$}};
    \fill[fill=black, outer sep=1mm]  (13) circle (0.1) node [below] {\textcolor{ForestGreen}{$I$}};
    \fill[fill=black, outer sep=1mm]  (14) circle (0.1) node [below] {\textcolor{ForestGreen}{$J$}};
    \fill[fill=black, outer sep=1mm]  (15) circle (0.1) node [below] {\textcolor{red}{$E$}};
    \fill[fill=black, outer sep=1mm]  (16) circle (0.1) node [below] {\textcolor{red}{$F$}};
    \fill[fill=black, outer sep=1mm]  (17) circle (0.1) node [below] {\textcolor{ForestGreen}{$K$}};
    \fill[fill=black, outer sep=1mm]  (18) circle (0.1) node [below] {\textcolor{red}{$G$}};
    \fill[fill=black, outer sep=1mm]  (19) circle (0.1) node [below] {\textcolor{red}{$H$}};
    \fill[fill=black, outer sep=1mm]  (20) circle (0.1) node [below] {\textcolor{ForestGreen}{$L$}};
    \fill[fill=black, outer sep=1mm]  (21) circle (0.1) node [below] {\textcolor{ForestGreen}{$I$}};
    \fill[fill=black, outer sep=1mm]  (22) circle (0.1) node [below] {\textcolor{ForestGreen}{$J$}};
    \fill[fill=black, outer sep=1mm]  (23) circle (0.1) node [below] {\textcolor{ForestGreen}{$K$}};
    \fill[fill=black, outer sep=1mm]  (24) circle (0.1) node [below] {\textcolor{ForestGreen}{$L$}};

\end{tikzpicture}
}

\end{center}

\caption{Greedy decomposition of a landscape into waves. The leftmost edges of the waves, forming a line, are bold.}
\label{ESZ4}
		
\end{figure}
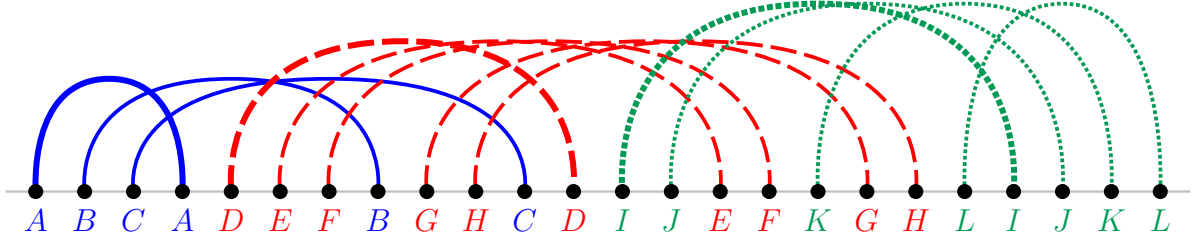

If $k\geqslant \ell+1$, then we are done. Otherwise,  $k\leqslant \ell$, and, since $p=\ell w+1$, some wave $W_i$ must have at least
$$\frac pk=\frac{\ell w+1}\ell>w$$
edges. This completes the proof.
\end{proof}

\section{An alternative proof of Lemma~\ref{lemma:span_prob}}\label{lemma:span_prob:2}

For simplicity we assume that $t\to\infty$ together with $n$.
Let $\mathcal E$ be the event that no edge in $\rm^{(r)}_{n}$ spans the sets  $A_1,\dots,A_r$. We will  generate $\rm^{(r)}_{n}$ using a random permutation~$\pi$ of $[rn]$ such that the edges are defined as $e_i = \pi(\{(i-1)r+1,\dots,ir\})$ for $1\le i\le n$.
	
Let $n'=\lceil n/2\rceil$. First observe that $\left|\bigcup_{1\le i\le n'} e_i \cap A_j\right|$ has a hypergeometric distribution,
\[
\Pr\left(\left|\bigcup_{1\le i\le n'} e_i \cap A_j\right| = k \right) = \frac{\binom{|A_j|}{k}\binom{rn-|A_j|}{rn'-k}}{\binom{rn}{rn'}}
= \frac{\binom{t}{k}\binom{rn-t}{rn'-k}}{\binom{rn}{rn'}}.
\]
Thus, the expected size of the intersection is $trn' / rn \ge t/2$ and the Chernoff bound for the hypergeometric distribution (see, e.g., \cite[Theorem 2.10]{JLR}) together with the union bound  yields
\[
\Pr\left(\exists_{1\le j\le r}: \left|\bigcup_{1\le i\le n'} e_i \cap A_j\right| \ge 0.99t \right) \le \sum_{1\le j\le r}\Pr\left( \left|\bigcup_{1\le i\le n'} e_i \cap A_j\right|\ge0.99t\right)\le \exp(-\Omega(t)).
\]

Let $\mathcal{F}$ be the event that ``$\forall_{1\le j\le r}: \left|\bigcup_{1\le i\le n'} e_i \cap A_j\right| < 0.99t$''. 
Also, for $1\le i\le n'$,  let ${\mathcal E}_i$ be the event that $e_i$ does not span  $A_1,\dots,A_r$. Note that conditioning on $\mathcal{F}$ implies that even after revealing the first $(i-1)r$ positions of $\pi$ (yielding $e_1,\dots,e_{i-1}$), we still have at least $0.01t$ uncovered points in each $A_j$. This, in particular, implies that
\[
\Pr(\bar{{\mathcal E}_i} | {\mathcal E}_1,\dots,{\mathcal E}_{i-1},\mathcal{F}) \ge \frac{(0.01t)^r}{\binom{rn - (i-1)r}{r}}
\ge \frac{(0.01t)^r}{\binom{rn}{r}}
\ge \left( \frac{0.01t}{en} \right)^r.
\]
Then, by the chain formula,
\begin{align*}
\PP(\mathcal E|\mathcal{F})
&\le\PP\left(\bigcap_{i=1}^{n'}{\mathcal E}_i|\mathcal{F}\right)=\PP({\mathcal E}_1|\mathcal{F})\PP({\mathcal E}_2|{\mathcal E}_1, \mathcal{F})\cdots\PP\left({\mathcal E}_{n'}\Big|\bigcap_{i=1}^{n'-1}{\mathcal E}_i, \mathcal{F}\right)\\
&\le \left( 1- \left( \frac{0.01t}{en} \right)^r \right)^{n'}
\le \exp\left( -\left( \frac{0.01t}{en} \right)^r n'\right)
\le \exp\left( -\frac{1}{2}\left( \frac{0.01}{e} \right)^r \frac{t^r}{n^{r-1}}\right).
\end{align*}

Finally,
\begin{align*}
\PP(\mathcal E) &= \PP(\mathcal E|\mathcal{F}) \PP(\mathcal{F}) + \PP(\mathcal E|\bar{\mathcal{F}}) \PP(\bar{\mathcal{F}})
\le \PP(\mathcal E|\mathcal{F}) +  \PP(\bar{\mathcal{F}})\\
&\le \exp\left( -\frac{1}{2}\left( \frac{0.01}{e} \right)^r \frac{t^r}{n^{r-1}}\right) + \exp(-\Omega(t))
\le \exp\left(- \Omega\left(\frac{t^r}{n^{r-1}}\right)\right),
\end{align*}
since $t=O(n)$. \qed

\begin{rem}\label{B1}
A similar idea can be used in order to replace Talagrand's inequality by Chebyshev's inequality in the proof of Theorem~\ref{thm:random}. Here we sketch how to do it.

In the proof of Theorem~\ref{thm:random} rename sets $B_1,\dots,B_{rk}$ as $A_1^{(\ell)},\dots,A_{r}^{(\ell)}$ for $1\le \ell \le k$ in such a way that for any fixed $\ell$ this $r$-tuple corresponds to sets $B_j$'s for $v_j\in e_\ell$.
Let $\eps>0$ be an arbitrarily small constant. Observe that for any $\eps n \le s \le n/2$,
$\left|\bigcup_{1\le i\le s} e_i \cap A_j^{(\ell)}\right|$ has a hypergeometric distribution with the expected size of the intersection $(s/n)t$. Thus, due to Chernoff's inequality a.a.s.\ the intersection has size $(s/n)t+o(t)$, uniformly over all choices of $s,\ell, j$. Let $\mathcal{F}$ be the event that ``$\forall_{j,\ell,s}: \left|\left|\bigcup_{1\le i\le s} e_i \cap A_j^{(\ell)}\right| - (s/n)t \right| < t^{2/3}$''.
Define $X_\ell$ to be the indicator random variable equal 1 if for each $\eps n \le i\le n/2$ there exists some $1\le j\le r$ such that $e_i\cap A_j^{(\ell)} = \emptyset$.
Now, conditioning on $\mathcal{F}$, one can quite precisely estimate $\E(X_{\ell})$ and $\E(X_{\ell}X_{\ell'})$ for different $\ell\neq\ell'$, and finally apply Chebyshev's inequality to the random variable $\sum_{\ell=1}^k X_{\ell}$.

\end{rem}

\section{Lower bound for lines -- a special case of Theorem~\ref{thm:random}}\label{thm:lines}
Here we give a simpler proof of Theorem~\ref{thm:random} in the special case of lines.
\begin{thm}\label{lower_line}
	A random matching $\rm^{(r)}_{n}$ contains a.a.s.\ a line of size  $\Omega_r(n^{1/r})$.
\end{thm}

The idea behind the proof of Theorem \ref{lower_line}, employed already in \cite{DGR-match}, is to restrict oneself to short edges of $\rm^{(r)}_{n}$ and show that among them the number of pairs of edges \emph{not} forming an alignment is smaller than the total number of such edges. Hence, removing one edge of each such pair leaves a large line.
To make this idea work, we need an auxiliary result counting short edges of $\rm^{(r)}_{n}$. Define the \emph{length} of a subset $S=\{i_1<\dots<i_s\}$ of $[rn]$ as $d(S):=i_s-i_1$ (see Fig.\ \ref{length}). Clearly, for every $S$ of size $s$ we have $s-1\le d(S)\le rn-1$. Furthermore, let $f_s^{(r)}(n,m)$ be the number of $s$-element subsets $S$ of $[rn]$ with $d(S)\le m$. Then
 \begin{equation}\label{f}
 f_s^{(r)}(n,m)=\sum_{s-1\le d\le m} (rn-d) \binom{d-1}{s-2}=rn\binom m{s-1}-(s-1)\binom{m+1}s,
  \end{equation}
  since for every choice of vertices $i_1$ and $i_s$ such that $i_s-i_1=d$ (there are $rn-d$ such choices), there are $\binom{d-1}{s-2}$ choices for the remaining vertices of $S$. Moreover,
 $\sum_{s-1\le d\le m} \binom{d-1}{s-2} = \binom{m}{s-1}$, while $\sum_{s-1\le d\le m} d\binom{d-1}{s-2} = (s-1)\binom{m+1}{s}$.

\begin{figure}[ht]
\captionsetup[subfigure]{labelformat=empty}
\begin{center}

\scalebox{1}
{
\centering
\begin{tikzpicture}
[line width = .5pt,
vtx/.style={circle,draw,black,very thick,fill=black, line width = 1pt, inner sep=0pt},
]

    \node[vtx] (1) at (0,0) {};
    \node[vtx] (2) at (2,0) {};
    \node[vtx] (3) at (4,0) {};
    \node[vtx] (4) at (6,0) {};
    \node[vtx] (5) at (7,0) {};
    \node[vtx] (6) at (11,0) {};

    \draw[line width=0.3mm, color=lightgray]  (1) -- (6);

    \draw [line width=0.5mm, decorate, decoration = {calligraphic brace,raise=10pt, amplitude=15pt}] (2,-0.2) --  (7,-0.2);
    \node at (4.5,1) {$i_s-i_1=d$};
    \node at (4.9,-0.4) {$\ldots$};

    \fill[fill=black, outer sep=1mm]   (1) circle (0.1) node [below] {$1$};
    \fill[fill=black, outer sep=1mm]   (2) circle (0.1) node [below] {$i_1$};
    \fill[fill=black, outer sep=1mm]  (3) circle (0.1) node [below] {$i_2$};
    \fill[fill=black, outer sep=1mm]   (4) circle (0.1) node [below] {$i_{s-1}$};
    \fill[fill=black, outer sep=1mm]   (5) circle (0.1) node [below] {$i_s$};
    \fill[fill=black, outer sep=1mm]   (6) circle (0.1) node [below] {$rn$};

\end{tikzpicture}
}

\end{center}

\caption{An edge of length $d$.}
\label{length}
			
\end{figure}

\begin{lemma}\label{lemma:edges}
	Let a sequence $m=m(n)$ be given such that $n^{1-\frac{1}{r-1}} \ll m \ll n$.
	Then, a.a.s.\ the number of edges  $e\in\rm^{(r)}_{n}$ with length $d(e)\le m$ is equal to $\frac{m^{r-1}}{(rn)^{r-2}} (1+o(1))$.
\end{lemma}

\begin{proof} We are going to apply the second moment method.
	For each $r$-tuple $e\in\binom{[rn]}r$, let $I_e$ be the indicator random variable such that $I_e=1$   if $e \in \rm^{(r)}_{n}$ and $I_e=0$, otherwise. Clearly, using the online scheme, $\Pr(I_{e}=1)=1/\binom{rn-1}{r-1}$. Let $X = \sum I_{e}$, where the summation is taken over all $e=\{i_1<\dots<i_r\}$ with $d(e)\le m$. In other words, $X$ counts all edges in $\rm^{(r)}_{n}$ of length at most $m$.
	Observe that the number of summands in the definition of $X$ is exactly $f_r^{(r)}(n,m)$. As
 $m \ll n$, by \eqref{f} we have
	\[
	 f_r^{(r)}(n,m)= rn \binom{m}{r-1} (1+o(1))
	\]
	and so
	\[
	\E X =\sum\E I_e=\frac{f_r^{(r)}(n,m)}{\binom{rn-1}{r-1}}= \frac{rn \binom{m}{r-1} (1+o(1))}{\binom{rn-1}{r-1}}
	= \frac{m^{r-1}}{(rn)^{r-2}} (1+o(1)).
	\]
	By the assumption that $m\gg n^{1-\frac{1}{r-1}}$, we have $\E X\to\infty$.

To estimate the second moment, notice that by the online scheme of generating $\rm^{(r)}_n$,
	\begin{align}\label{ee}
	\Pr(I_{e}=I_{e'}=1) &=\Pr(e,e'\in \rm^{(r)}_n)=\Pr(e\in \rm^{(r)}_n)\Pr(e'\in \rm^{(r)}_n|e\in\rm^{(r)}_n)\notag \\
&=\begin{cases}
		\frac{1}{\binom{rn-1}{r-1} \binom{rn-r-1}{r-1}}, & \text{ if $e\cap e'=\emptyset$},\\
		0, & \text { otherwise}.
	\end{cases}
	\end{align}

	Thus, quite crudely,
\begin{align*}
	\E(X(X-1))&=\sum_{e:d(e)\le m}\sum_{e': e'\cap e=\emptyset, d(e')\le m}\Pr(I_e=I_{e'}=1)\\
	&\le \frac{(f_r^{(r)}(n,m))^2}{\binom{rn-1}{r-1}\binom{rn-r-1}{r-1}} = (\E X)^2 \frac{\binom{rn-1}{r-1}}{\binom{rn-r-1}{r-1}}.
\end{align*}

Let $\eps:=\eps(n)\to0$, but $\eps\gg\sqrt{n^{r-2}/m^{r-1}}$.
Then, by applying Chebyshev's inequality
	\begin{align*}\Pr(|X-\E X|\ge\eps\E X) &\le \frac{\E (X(X-1))+\E X-(\E X)^2}{\eps^2(\E X)^2}\\&\le\frac1{\eps^2}\left(\frac{\binom{rn-1}{r-1}}{\binom{rn-r-1}{r-1}}+\frac1{\E X}-1\right)=\frac1{\eps^2}\left(O\left(1/n\right)+\frac{1}{\E X}\right)
		\to0,
	\end{align*}
	as $\eps^2 \E X \to\infty$. Thus, a.a.s.\ $X=(1\pm\eps)\E X=\frac{m^{r-1}}{(rn)^{r-2}} (1+o(1))$.
\end{proof}


\begin{proof}[Proof of Theorem \ref{lower_line}]
	For $r=2$, it was already proved in~\cite[Theorem 12]{DGR-match} that a.a.s.\ in $\rm^{(2)}_{n}$ there are lines of size $\Omega(\sqrt{n})$. Nevertheless, we include this case here.
	
	Let  $m =  n^{1-\frac{1}{r}}$ (for simplicity assume that this is an integer). Note that $n^{1-\frac{1}{r-1}} \ll m \ll n$. By Lemma~\ref{lemma:edges}, a.a.s.\ the number of edges $X$ of length at most~$m$ in $\rm^{(r)}_{n}$ is
	\[
X=\frac{m^{r-1}}{(rn)^{r-2}} (1+o(1))
	\ge\frac{2m^{r-1}}{3(rn)^{r-2}}
	= \frac{2n^{(r-1)^2/r}}{3(rn)^{r-2}}
	= \frac{2}{3r^{r-2}}n^{1/r}.
	\]
	We say that two edges $\{i_1<\dots<i_r\}$ and $\{j_1<\dots<j_r\}$ form a \emph{nonliner} if either $i_1<j_1<i_r$ or $j_1<i_1<j_r$.  In other words, a nonliner is any pattern, collectable or not, other than the alignment $P_1$.
	
	We are going to show that among the edges of length at most~$m$, there are a.a.s.\ at most $\frac{1}{2r^{r-2}}n^{1/r}$ nonliners. After removing one edge from each nonliner (possibly with repetitions) we then obtain a line of size at least $\frac{1}{6r^{r-2}}n^{1/r}$, which will end the proof.
	
	For a $2r$-element subset $T=\{i_1<\dots<i_r,\; j_1<\dots<j_r\}\subset [rn]$ with $i_1<j_1<i_r$, let $I_T$ be the indicator random variable equal to $1$ if both $e_1:=\{i_1<\dots<i_r\}$ and $e_2:=\{j_1<\dots<j_r\}$ are edges of $\rm^{(r)}_n$,  and $I_T=0$ otherwise. Note that if $I_T=1$, then $e_1,e_2$ form a nonliner  in $\rm^{(r)}_n$. By \eqref{ee},
	\[
	\Pr(I_T=1)=\frac{1}{\binom{rn-1}{r-1} \binom{rn-r-1}{r-1}}
	= \frac{((r-1)!)^2}{(rn)^{2r-2}} (1+o(1)).
	\]

\begin{figure}[ht]
\captionsetup[subfigure]{labelformat=empty}
\begin{center}

\scalebox{1}
{
\centering
\begin{tikzpicture}
[line width = .5pt,
vtx/.style={circle,draw,black,very thick,fill=black, line width = 1pt, inner sep=0pt},
]

    \node[vtx] (1) at (0,0) {};
    \node[vtx] (2) at (2,0) {};
    \node[vtx] (3) at (10,0) {};
    \node[vtx] (4) at (5,0) {};
    \node[vtx] (5) at (7,0) {};
    \node[vtx] (6) at (11,0) {};

    \draw[line width=0.3mm, color=lightgray]  (1) -- (6);

    \draw [line width=0.5mm, decorate, decoration = {calligraphic brace,raise=10pt, amplitude=15pt}] (2,-0.2) --  (7,-0.2);
    \node at (4.5,1) {$i_r-i_1\le m$};
    \draw [line width=0.5mm, decorate, decoration = {calligraphic brace,mirror,raise=10pt, amplitude=15pt}] (5,-0.35) --  (10,-0.35);
    \node at (7.5,-1.5) {$j_r-j_1\le m$};
    \node at (3.5,-0.4) {$\ldots$};
    \node at (5.9,-0.4) {$\ldots$};
    \node at (8.5,-0.4) {$\ldots$};

    \fill[fill=black, outer sep=1mm]   (1) circle (0.1) node [below] {$1$};
    \fill[fill=black, outer sep=1mm]   (2) circle (0.1) node [below] {$i_1$};
    \fill[fill=black, outer sep=1mm]  (3) circle (0.1) node [below] {$j_r$};
    \fill[fill=black, outer sep=1mm]   (4) circle (0.1) node [below] {$j_1$};
    \fill[fill=black, outer sep=1mm]   (5) circle (0.1) node [below] {$i_r$};
    \fill[fill=black, outer sep=1mm]   (6) circle (0.1) node [below] {$rn$};

\end{tikzpicture}
}

\end{center}

\caption{A nonliner with edges of lengths at most $m$.}
\label{T}
			
\end{figure}

	Let $Y = \sum I_T$, where the summation is taken over all sets $T$ as above and such that $i_r-i_1\le m$ and $j_r-j_1\le m$  (see Fig.\ \ref{T}). Then $Y$ counts all nonliners in $\rm^{(r)}_n$ formed by the edges of length at most $m$. Let $g^{(r)}(n,m)$ denote the number of terms in this sum. In order to estimate $g^{(r)}(n,m)$,  we first count the number of choices of $e_1$ and $j_1$. Since $i_1<j_1<i_r$, we can treat them as one  $(r+1)$-set $\{\iota_1<\dots<\iota_{r+1}\}$ from $[rn]$ of length at most $m$. Once it is chosen, we designate one of the vertices in~$\{\iota_2,\dots,\iota_{r}\}$ to be $j_1$ and the remaining ones become $e_1$. Using notation from prior to Lemma~\ref{lemma:edges}, the number of such choices is
	\[
	f_{r+1}^{(r)}(n,m) \cdot (r-1)
	= r(r-1)n \binom{m}{r} (1+o(1)).
	\]
	Finally, the number of choices of $\{j_2,\dots,j_r\}$ satisfying $j_r-j_1\le m$ is at least $\binom{m-(r-1)}{r-1}$ and at most $\binom{m}{r-1}$. Thus,
	\[
	g^{(r)}(n,m) = r(r-1)n \binom{m}{r} \binom{m}{r-1} (1+o(1))
	= \frac{(r-1)nm^{2r-1}}{((r-1)!)^2} (1+o(1))
	\]
	and so
	\[
	\E Y = \frac{(r-1)nm^{2r-1}}{(rn)^{2r-2}} (1+o(1))
	= \frac{(r-1)nn^{(r-1)(2r-1)/r}}{(rn)^{2r-2}} (1+o(1))
	= \frac{r-1}{r^{2r-2}} n^{1/r} (1+o(1)).
	\]
	It remains to show that, say, a.a.s.\ $Y\le \frac 32\E Y$. If so, then a.a.s.\
	\[
	Y \le \frac{3(r-1)}{2r^{2r-2}} n^{1/r}(1+o(1))  < \frac{1}{2r^{r-2}}n^{1/r},
	\]
as required,
	since the inequality $\frac{3(r-1)}{2r^{2r-2}} < \frac{1}{2r^{r-2}}$ is equivalent to $3(r-1)< r^r$, which holds for every $r\ge2$.
	
	To achieve our last goal, we apply Chebyshev's inequality. For this sake, we need to estimate $\E (Y(Y-1))$, which can be written as
	\[
	\E(Y(Y-1))=\sum_{T,T'}\Pr(I_T=I_{T'}=1),
	\]
	where the summation is taken over all (ordered) pairs of potential
	  nonliners in $\rm^{(r)}_n$: $T=(e_1,e_2)$ and $T'=(e_1',e_2')$. We split the above sum into two sub-sums $\Sigma_1$ and $\Sigma_2$ according to whether $T\cap T'=\emptyset$ or $|T\cap T'|=r$ (and then $e_i=e_j'$ for some $1\le i\le j\le 2$ -- for all other options the above probability is zero).
	
	In the former case, by a similar chain formula as in \eqref{ee},
$$\Pr(I_T=I_{T'}=1)=\frac1{\binom{rn-1}{r-1} \binom{rn-r-1}{r-1}\binom{rn-2r-1}{r-1} \binom{rn-3r-1}{r-1}}$$
and so

	\[
	\Sigma_1\le\frac{g^{(r)}(n,m)^2}{\binom{rn-1}{r-1} \binom{rn-r-1}{r-1}\binom{rn-2r-1}{r-1} \binom{rn-3r-1}{r-1}}
	=(\E Y)^2(1+o(1)).
	\]
	In the latter case,
$$\Pr(I_T=I_{T'}=1)=\frac1{\binom{rn-1}{r-1} \binom{rn-r-1}{r-1}\binom{rn-2r-1}{r-1} },$$
while the number of such pairs $(T,T')$ is at most $g^{(r)}(n,m)\cdot 4m\binom{m}{r-1}$, as given $T$, there are four ways to select the common $r$-tuple and at most $m\binom{m}{r-1}$ ways to select the other $r$-tuple of $T'$. Thus,
	\[
	\Sigma_2\le \frac{g^{(r)}(n,m)\cdot 4m\binom{m}{r-1}}{\binom{rn-1}{r-1} \binom{rn-r-1}{r-1}\binom{rn-2r-1}{r-1}}
	=O_r\left(\frac{nm^{2r-1}\cdot m^r}{n^{3r-3}} \right)
	=O_r\left(\frac{m^{3r-1}}{n^{3r-4}} \right)
	=O_r(n^{1/r})
	\]
	and, altogether,
	\[
	\E (Y(Y-1)) \le (\E Y)^2(1+o(1)) + O_r(n^{1/r}).
	\]
	By Chebyshev's inequality,
	\begin{align*}
		\Pr(|Y-\E Y|\ge\E Y/2)&\le \frac{\E(Y(Y-1))+\E Y-(\E Y)^2}{(\E Y/2)^2}\\
		&\le\frac14\left(1+o(1)+\frac{O_r(n^{1/r})}{(\E Y)^2}+\frac1{\E Y}-1\right)=O_r(n^{-1/r})+o(1)=o(1).
	\end{align*}
	Thus, a.a.s.\ $Y\le \frac32\E Y$, as required.
\end{proof}

\begin{rem}
In fact, for $r=2$, as shown in~\cite{DGR-match}, it is sufficient to prove the lower bound in Theorem \ref{thm:random} only for lines. Indeed, then the lower bound for the other two patterns follows from Lemma~\ref{lem:upbr} combined with Corollary \ref{Theorem ESz General_partite} applied to the bipartite sub-matching of $\rm^{(2)}_n$, that is, one consisting  of edges with one endpoint in $[n]$ and the other in $[2n]\setminus[n]$ (note that such a sub-matching has no alignments and has, a.a.s., about $n/2$ edges; for details, see \cite{DGR-match}).
\end{rem}

\begin{rem}
In the above proof we had a little choice in defining parameter $m$. Indeed, there were two constraints: 1) the number of short edges $X=\Theta_r\left( \frac{m^{r-1}}{n^{r-2}}\right)$ should be $\Omega_r(n^{1/r})$ and 2) the number of nonliners $Y=\Theta_r\left( \frac{m^{2r-1}}{n^{2r-3}}\right)$ should be $O_r\left( \frac{m^{r-1}}{n^{r-2}}\right).$
And they are equivalent, respectively, to $m=\Omega_r(n^{1-1/r})$ and $m=O_r(n^{1-1/r})$.
\end{rem}

\end{document}